\documentclass[10pt,a4paper]{amsart}

\usepackage[latin1]{inputenc}
\usepackage[english]{babel}
\usepackage[T1]{fontenc}
\usepackage{amsmath}
\usepackage{indentfirst}
\usepackage{multirow}
\usepackage{color} 
\usepackage{array}
\usepackage{multicol}
\usepackage{hyperref}
\usepackage{amsfonts}
\usepackage{amsthm}
\usepackage{amssymb}
 
\newtheorem{teo}{Theorem}[section]
\newtheorem{lema}{Lemma}[section]

\newtheorem{prop}[teo]{Proposition}

\newtheorem{defi}{Definition}[section]

\newcommand{\ffrac}{\displaystyle\frac}

\def\s{\operatorname{\textbf{Sgn}}}
\def\R{\mathbb{R}}
\def\vizp{\mathcal{U}_p}
\def\l{\ell}

\DeclareMathOperator*{\mycup}{\cup}

\begin{document}

\title{A new proof of a theorem of Dutertre and Fukui on Morin singularities}

\author{Camila M. \textsc{Ruiz}}

\address{Universidade de São Paulo\\
Instituto de Ciências Matemática e de Computação \\
Avenida Trabalhador São-carlense, 400 - Centro\\
CEP: 13566-590 - São Carlos - SP, Brasil}
\email{cmruiz@icmc.usp.br}

%\keywords{Morin singularities, Morse Theory, Euler characteristic}
\maketitle

\begin{abstract}
In \cite{Dutertrefukui}, N.Dutertre and T. Fukui used Viro's integral calculus to study the topology of stable maps $f:M\rightarrow N$ between two smooth manifolds $M$ and $N$. 
They also discussed several applications to Morin maps. 
In particular, in Theorem 6.2 \cite{Dutertrefukui}, they show an equality relating the Euler characteristic of a compact manifold $M$ and the Euler characteristic of the singular sets of a Morin map defined on $M$. 
%In this work, we present a new proof of such theorem when $N=\R^n$, using not Viro's integral calculus, but Morse Theory for manifolds with boundary.

In this paper we show how Morse theory for manifolds with boundary can be applied to the study of the singular sets of a Morin map in order to obtain a new proof of Dutertre-Fukui's Theorem when $N=\R^n$.
\end{abstract}

%%%%%%%%%%%%%%%%%%%%%%%%%%%%%%%%%%%%%%%%%%%%%%%%%%%%%%%%  Introduction %%%%%%%%%%%%%%%%%%%%%%%%%%%%%%%%%%%%%%%%%%%%%%%%%%%%%%%%%%%%%%%

\section{Introduction}

Let $f:M\rightarrow\R^n$ be a $C^{\infty}$ map defined on a compact manifold of dimension $m$ with $m\geq n$ having only Morin singularities. In \cite{Fukuda}, T. Fukuda used Morse theory to study the topology of the singular sets of $f$ and their relation with the topology of $M$. The author has obtained a modulo 2 congruence formula associating the Euler characteristic of $M$ to the Euler characteristic of the singular sets $\overline{A_k(f)}$: \begin{equation}\label{Fukuda1}\chi(M)\equiv \displaystyle\sum_{k=1}^{n}\chi(\overline{A_k(f)}) \mod2,\end{equation} where $A_k(f)$ is the set of points $x$ in $M$ such that $f$ has a singular point of type $A_k$ at $x$. Furthermore, in the case that $f$ admits only fold singular points, that is, Morin singular points of type $A_1$, T. Fukuda proved an equality relating $\chi(M)$ to the singular set of $f$: \begin{equation}\label{Fukuda2}\chi(M)=\chi(A_1^+(f))-\chi(A_k^-(f)).\end{equation} Later O. Saeki \cite{Saeki} used Morse theory in the sense of Bott \cite{Bott} to extend T. Fukuda's formulas to the case that $f:M\rightarrow N$ is a Morin map with $\dim M\geq\dim N$.

Recently, N. Dutertre and T. Fukui \cite{Dutertrefukui} investigated how Viro's integral calculus can be applied to study the topology of stable maps $f:M\rightarrow N$ and established relations between the Euler characteristic of $M$ and $N$ and the singular sets of $f$. Applying the obtained results to the context of Morin maps and using the link between the Euler characteristic with closed support and the topological Euler characteristic, the authors could recover and improve several  well-known results, including some formulas obtained by O. Saeki in \cite{Saeki} and the formulas (\ref{Fukuda1}) and (\ref{Fukuda2}) obtained by T. Fukuda in \cite{Fukuda}. In particular, N. Dutertre and T. Fukui gave a characterization for the subsets $A_k^+(f)$ and $A_k^-(f)$ of $A_k(f)$ and proved the following theorem.

\begin{teo}\cite[Theorem 6.2]{Dutertrefukui}\label{DutertreFukuiRn}
Let $f:M\rightarrow N$ be a Morin map. Assume that $M$ is a $m$-dimensional compact manifold, $N$ is a $n$-dimensional connected manifold and $m-n$ is odd. Then, $$\chi(M)=\sum_{k: \text{ odd}}{\left[\chi(\overline{A_k^+(f)})-\chi(\overline{A_k^-(f)})\right]}.$$
\end{teo}
 
In this paper we present a new proof of Dutertre-Fukui's theorem when $N=\R^n$. Inspired by the papers of T. Fukuda and O. Saeki, we investigate how Morse theory for manifolds with boundary can be applied to study the singular sets of a Morin map in order to obtain the new proof. In addition to Morse theory for manifolds with boundary, the new proof version uses the local forms of Morin singularities and the characterization of $A_k^+(f)$ and $A_k^-(f)$ given in \cite{Dutertrefukui}. 

In Section \ref{lemmasFukuda}, we present lemmas proved by T. Fukuda in \cite{Fukuda} that will be used in the new proof of Dutertre-Fukui's Theorem. In Section \ref{DutertreFukuiTheorem}, we recall the main results of N. Dutertre and T. Fukui's paper. In Section \ref{MorseTheory} we briefly discuss results of Morse theory for manifolds with boundary that will be used in Sections \ref{correctpts} and \ref{newproof}. 

We consider generic projections $L_a:\R^n\rightarrow\R$ given by $L_a(x)=\sum_{i=1}^{n}a_ix_i$, such that if $f$ is a Morin map, then $L_a\circ f$ are Morse functions with good properties that satisfy lemmas by T. Fukuda presented in Section \ref{lemmasFukuda}. Then, in Section \ref{correctpts} we apply T. Fukuda's lemmas and Morse theory for manifolds with boundary to study the correct critical points of the restrictions $L_a\circ f|_{\overline{A_k(f)}}$ to the singular sets $\overline{A_k(f)}$ of $f$. Finally, we conclude this paper with the new proof of Dutertre-Fukui's Theorem in Section \ref{newproof}.

The author would like to express her sincere gratitude to Nicolas Dutertre and Nivaldo de Góes Grulha Júnior for fruitful discussions and valuable comments that resulted in this work.
The author was supported by CNPq, "Conselho Nacional de Desenvolvimento Científico e Tecnológico", Brazil (grants 143479/2011-3 and 209531/2014-2).

%%%%%%%%%%%%%%%%%%%%%%%%%%%%%%%%%%%%%%%%%%%%%%%%%%%%% T. Fukuda's Lemmas  %%%%%%%%%%%%%%%%%%%%%%%%%%%%%%%%%%%%%%%%%%%%%%%%%%%%%%%%%%%%
\section{T. Fukuda's lemmas}\label{lemmasFukuda}

Let $M$ be a compact manifold of dimension $m$, $N$ be a connected manifold of dimension $n$ with $m\geq n$ and $f:M\rightarrow N$ be a $C^{\infty}$ map. We say that $p\in M$ is a Morin singular point of type $A_k$ $(k=1, \ldots,n)$ of $f$ if there exists local coordinates $x=(x_1, \ldots, x_m)$ around $p$ and $y=(y_1, \ldots, y_n)$ around $f(p)$ such that $f$ may be written locally as: 
\begin{equation}\label{formalocalcap2}
\begin{array}{l}
 y_i\circ f=x_i, \text{ for } i\leq n-1;\\
 y_n\circ f=x_n^{k+1}+\displaystyle\sum_{i=1}^{k-1}x_ix_n^{k-i}+ x_{n+1}^2+\ldots+ x_{n+\lambda-1}^2-x_{n+\lambda}^2-\ldots-x_m^2.
\end{array}
\end{equation} We say that $f$ is a Morin map if $f$ admits only Morin singular points. We denote by $A_k(f)$ the set of Morin singular points of type $A_k$ of $f$ and by $\overline{A_k(f)}$ the topological closure of $A_k(f)$ and we set $A_0(f)$ the set of regular points of $f$. The properties presented in the following lemma are well known (see Lemma 2.2 in \cite{Fukuda}).

\begin{lema}\label{propriedadesak} Let $f:M\rightarrow\R^n$ be a Morin map as above. Then we have:
\begin{enumerate}
\item $A_k(f)$ and $\overline{A_k(f)}$ are submanifolds of $M$ having dimension $n-k$, for $k=1, \ldots, n$;
\item $\overline{A_k(f)}=\displaystyle\mycup\limits_{i\geq k}A_i(f)$, for $k=0, \ldots, n$.
\end{enumerate}
\end{lema}

Let $L\in\R P^{n-1}$ be a straight line in $\R^n$ and let $\pi_L:\R^n\rightarrow L$ be the orthogonal projection to $L$. In \cite{Fukuda}, T. Fukuda apply Morse theory and well known properties of singular sets $A_k(f)$ of a Morin map $f:M\rightarrow\R^n$ to study the critical points of mappings $\pi_L\circ f:M\rightarrow L$ and their restrictions to the singular sets $\pi_L\circ f|_{A_k(f)}$ and $\pi_L\circ f|_{\overline{A_k(f)}}$. In \cite{Saeki}, O. Saeki has extended the results obtained by T. Fukuda to the case of a Morin map $f:M\rightarrow N$, where $N$ is a manifold with $\dim M\geq\dim N$. 

In this work we will consider linear maps $L_a:\R^n\rightarrow\R$ defined by $L_a(x)=\sum_{i=1}^{n}a_ix_i$, for $a=(a_1, \ldots, a_n)\in\R^n\setminus\{\vec{0}\}$. These mappings are more general than the orthogonal projections $\pi_L:\R^n\rightarrow L$ since, up to isomorphism, all $\pi_L$ can be written in the form $L_a(x)=\sum_{i=1}^{n}a_ix_i$ for some $a\in\R^n\setminus\{\vec{0}\}$ with $\left\|a\right\|=1$. Moreover, $L_a\circ f$ and its restrictions $L_a\circ f|_{ A_k(f)}$ and $L_a\circ f|_{ \overline{A_k(f)}}$ also satisfy the following lemmas presented by T. Fukuda to study $\pi_L\circ f$. For details, see \cite{Fukuda} and \cite{Ruiz}.\\

Let $f:M^m\rightarrow\R^n$ be a $C^{\infty}$ Morin map with $m\geq n$. The following lemma holds for every $a\in\R^n\setminus\{\vec{0}\}$.

\begin{lema}\label{lemaseparado} For a point $p\in A_{k+1}(f)$, $p$ is a critical point of $L_a\circ f|_{A_{k+1}(f)}$ if and only if $p$ is a critical point of $L_a\circ f|_{\overline{A_k(f)}}$.
\end{lema}

In the following, a property is said to hold for almost every $a\in\R^n\setminus\{\vec{0}\}$ if the set of $a$ for which it does not hold is of Lebesgue measure zero in $\in\R^n\setminus\{\vec{0}\}$. Moreover, we denote by $C(g)$ the set of critical points of a map $g:M\rightarrow\R$.

%\begin{lema} For almost every $a\in\R^n\setminus\{\vec{0}\}$, the restricted mappings $L_a\circ f|_{A_{k}(f)}$ are Morse functions, for $k=1,\ldots,n$.
%\end{lema}

\begin{lema} \label{Laproperties} For almost every $a\in\R^n\setminus\{\vec{0}\}$, the linear mapping $L_a:\R^n\rightarrow\R$ satisfies the properties:
\begin{enumerate}
\item The restricted mappings $L_a\circ f|_{A_{k}(f)}$ are Morse functions, for $k=1,\ldots,n$;
\item $C(L_a\circ f|_{\overline{A_k(f)}})\cap \overline{A_{k+2}(f)}=\emptyset$, $k=0,\ldots,n-2$;
\item For a point $q\in A_{k+1}(f)$, $q$ is a non-degenerated critical point of $L_a\circ f|_{A_{k+1}(f)}$ if, and only if, $q$ is a non-degenerated critical point of $L_a\circ f|_{\overline{A_k(f)}}$;
\item Every point of $A_n(f)$ is a non-degenerated critical point of $L_a\circ f|_{\overline{A_{n-1}(f)}}$;
\item $L_a\circ f$ and $L_a\circ f|_{\overline{A_{k}(f)}}$ are Morse functions, for $k=1,\ldots,n$.
\end{enumerate}
\end{lema}

%\begin{lema}\label{lemamorsecomplementar} For almost every $a=(a_1,\ldots, a_n)\in\R^n\setminus\{\vec{0}\}$, $L_a\circ f$ and $L_a\circ f|_{\overline{A_{k}(f)}}$ are Morse functions.
%\end{lema}

Suppose that $p\in A_1(f)$, then $y_n\circ f$ in (\ref{formalocalcap2}) has the form $$y_n\circ f=x_n^{2}+\ldots+ x_{n+\lambda-1}^2-x_{n+\lambda}^2-\ldots-x_m^2.$$ If $m-n$ is odd, then $\lambda\equiv m-n-\lambda+1 \mod2$ and the parity of the Morse index of the Morse function $y_n\circ f$ does not depend on the choice of the coordinates. Hence, the sets $$\begin{array}{c} A_{1}^{+}(f) = \{p\in M: p\in A_{k}(f) \text{ with } \lambda \equiv 0 \mod 2\};\\
 A_{1}^{-}(f) = \{p\in M: p\in A_{k}(f) \text{ with } \lambda \equiv 1 \mod 2\};
\end{array}$$ are well-defined.

\begin{lema}\label{indexlemma}
Suppose that $m-n+1$ is even. Let $q\in A_1(f)$ be a critical point of $L_a\circ f$, then  $$\begin{array}{llll}
Ind(L_a\circ f,q)\equiv & Ind(L_a\circ f|_{A_1^+(f)},q) & \mod2, & \text{ if } q\in A_1^+(f);\\ 
Ind(L_a\circ f,q)\equiv & 1 + Ind(L_a\circ f|_{A_1^-(f)},q) & \mod2, & \text{ if } q\in A_1^-(f); 
\end{array}$$ where $Ind(g,x)$ denotes the Morse index of a critical point $x$ of a Morse function $g$.
\end{lema}

%In Section \ref{newproof}, we will use these lemmas in the new proof of Dutertre-Fukui's Theorem.

%%%%%%%%%%%%%%%%%%%%%%%%%%%%%%%%%%%%%%%%%%%%%%%%%%%  Dutertre-Fukui's Theorem  %%%%%%%%%%%%%%%%%%%%%%%%%%%%%%%%%%%%%%%%%%%%%%%%%%%%%%%

\section{Dutertre-Fukui's Theorem}\label{DutertreFukuiTheorem}

Let $M$ be a compact manifold of dimension $m$, $N$ a connected manifold of dimension $n$ and $f:M\rightarrow N$ a $C^{\infty}$ Morin map defined on $M$, with $m>n$ and $m-n$ odd.\\

Note that, if $m-n$ is odd, then $\lambda\equiv m-n-\lambda+1 \mod2$. If $k$ is odd, N. Dutertre and T. Fukui introduced in \cite{Dutertrefukui} a definition for the subsets $A_k^+(f)$ and $A_k^-(f)$ of $A_k(f)$ given in terms of the parity of the Morse index $\lambda$ of the quadratic part of the function $y_n\circ f$ in (\ref{formalocalcap2}):

\begin{defi}\cite[p. 186]{Dutertrefukui} Let $k$ be odd, then $$\renewcommand{\arraystretch}{1.5}{\begin{array}{c} A_{k}^{+}(f) = \{p\in M: p\in A_{k}(f) \text{ with } \lambda \equiv 0 \mod 2\};\\
 A_{k}^{-}(f) = \{p\in M: p\in A_{k}(f) \text{ with } \lambda \equiv 1 \mod 2\}.
\end{array}}$$
\end{defi}

It is well known that for $k\geq1$, the $A_k(f)$'s and $\overline{A_k(f)}$'s are smooth manifolds of dimension $(n-k)$ such that $\overline{A_k(f)}=\displaystyle\mycup\limits_{i=1}^{n}A_i(f)$ (see \cite{Fukuda}). 
In \cite{Dutertrefukui}, it is shown a similar result for the subsets $A_k^+(f)$ and $A_k^-(f)$. That is:

\begin{prop}\cite[Proposition 6.1]{Dutertrefukui} \label{PropDF} If $k$ is odd, then $\overline{A_{k}^{+}(f)}$ and $\overline{ A_{k}^{-}(f)}$ are compact manifolds with boundary of dimension $n-k$. 
Moreover, $\partial \overline{ A_{k}^{+}(f)}=\partial \overline{ A_{k}^{-}(f)}=\overline{A_{k+1}(f)}.$
\end{prop}

Finally, N. Dutertre and T. Fukui used integral calculus due to O. Viro \cite{Viro} to prove the following theorem that is an improvement of a result of T. Fukuda \cite{Fukuda} for $N=\R^n$ and O. Saeki \cite{Saeki} for a general $N$:

\begin{teo}\cite[Theorem 6.2]{Dutertrefukui}\label{DutertreFukuiRn}
Let $f:M\rightarrow N$ be a Morin map. Assume that $M$ is a $m$-dimensional compact manifold, $N$ is a $n$-dimensional connected manifold and $m-n$ is odd. Then, $$\chi(M)=\sum_{k: \text{ odd}}{\left[\chi(\overline{A_k^+(f)})-\chi(\overline{A_k^-(f)})\right]}.$$
\end{teo}

%%%%%%%%%%%%%%%%%%%%%%%%%%%%%%%%%%%%%%%%%%%% Morse Theory to manifolds with boundary %%%%%%%%%%%%%%%%%%%%%%%%%%%%%%%%%%%%%%%%%%%%%%%%

\section{Morse theory for manifolds with boundary}\label{MorseTheory}

In this section we recall the main results on Morse theory for manifolds with boundary presented in \cite[Paragraph 3]{HammLe}.\\

Let $(M,\partial M)$ be a compact smooth manifold with boundary of dimension $m$ and let $f:M\rightarrow\R$ be a $C^{\infty}$ function. Let us call $\partial f$ the restricted map of $f$ to the boundary $\partial M$, $f|_{\partial M}$, and $f^{\circ}$ the restricted map of $f$ to ${M\setminus\partial M}$, $f|_{M\setminus\partial M}$. The critical points of $f$ consist of the critical points of $f^{\circ}$ and the critical points of $\partial f$.

Let $q\in \partial M$ be a critical point of $\partial f$. Let $(U,h)$ be a chart in a neighborhood of $q$, that is, $U$ is a neighborhood of $q$ in M, $h$ is a $C^{\infty}$ diffeomorphism of $U$ to $B_{\varepsilon}^{-}$, which is defined by $B_{\varepsilon}^{-}=\{x\in\R^n|\sum_{i=1}^{n}x_i^{2}<\varepsilon \text{ and } x_n\leq0\}$, such that $h(q)=0$. The function $f\circ h^{-1}$ is $C^{\infty}$ on $B_{\varepsilon}^{-}$ and can be extended to a function $\tilde{f}$ on $B_{\varepsilon}^{\circ}=\{x\in\R^n|\sum_{i=1}^{n}x_i^{2}<\varepsilon\}$.

\begin{defi}
We say that $q\in\partial M$ is a correct critical point of $f$ if $q$ is a critical point of $\partial f$ and 0 is not a critical point of $\tilde{f}$. We say that $q\in\partial M$ is a non-degenerate correct critical point of $f$ if $q$ is a correct critical point of $f$ and a non-degenerate critical point of $\partial f$.
\end{defi}

The tangent space to $B_{\varepsilon}^{\circ}$ at $0$ is $\R^n$ and the kernel of the derivative $d\tilde{f}(0)$ is the hyperplane $x_n=0$. This hyperplane divides $\R^n$ in two open half-spaces where $d\tilde{f}(0)$ does not vanish and takes opposite signs. Therefore, the correct critical points of $f$ are the critical points $q$ for which the derivative $df(q)$ vanishes on the tangent space of the boundary of $M$ at $q$, $T_q(\partial M)$, but does not vanish in whole tangent space of the manifold $M$, $T_qM$.

\begin{defi} Let $M$ be provided with a Riemannian metric. If $q$ is a correct critical point of $f$ then:
\begin{itemize}
\item If the sign of $d\tilde{f}(0)$ in the half-space defined by $x_n>0$ is negative, we say that $q$ is a critical point with gradient of $f$ ``pointing inwards''.
\item If the sign of $d\tilde{f}(0)$ in the half-space defined by $x_n>0$ is positive, we say that $q$ is a critical point with gradient of $f$ ``pointing outwards''.
\end{itemize} 
\end{defi}

\begin{defi}A $C^{\infty}$ function $f:M\rightarrow\R$ is a correct function if every critical point of $\partial f$ is a correct critical point of $f$. A $C^{\infty}$ function $f:M\rightarrow\R$ is a correct Morse function if $f^{\circ}$ and $\partial f$ admit only non-degenerate critical points and if $f$ admits only correct critical points on $\partial M$.
\end{defi}

%\begin{obs} In this work, when we consider the gradient $\nabla f(p)$ of a map $f:M\rightarrow\R$ at a point $p$, we will be  actually considering the orthogonal projection of $\nabla f(p)$ on $T_pM$.
%\end{obs}

\begin{prop}\label{abertodenso} Let $M$ be a $C^{\infty}$ compact manifold with boundary, correct Morse functions form a dense open subset of $C^{\infty}(M,\R)$.
\end{prop}

\begin{teo}\label{teobordo}
Let $M$ be a $C^{\infty}$ compact manifold with boundary of dimension $m$ provided with a Riemannian metric and let $f:M\rightarrow\R$ be a correct Morse function. Let $p_1,\ldots,p_n$ be the critical points of $f^{\circ}$ and $\lambda_1,\ldots,\lambda_n$ their respective Morse indices and let $q_1,\ldots,q_{\l}$ be the correct critical points of $f$ and $\mu_1,\ldots,\mu_{\l}$ their respective Morse indices. Then: $$\chi(M)=\displaystyle\sum_{i=1}^{n}(-1)^{\lambda_i}+\displaystyle\sum_{q_j|\nabla f \text{ \textit{pointing inwards}}}(-1)^{\mu_j}.$$
\end{teo}

%%%%%%%%%%%%%%%%%%%%%%%%%%%%%%%%%%%%%%%%%%%%%%%%%%   Correct critical points  %%%%%%%%%%%%%%%%%%%%%%%%%%%%%%%%%%%%%%%%%%%%%%%%%%%%%%%

\section{Correct critical points}\label{correctpts}

In order to prove Theorem \ref{DutertreFukuiRn}, we have to study the sum $$\displaystyle\sum_{k: \text{ odd}}{\left[\chi(\overline{A_k^+(f)})-\chi(\overline{A_k^-(f)})\right]},$$ to do this, we will analyse the following expression $$\chi(\overline{A_k^+(f)})-\chi(\overline{A_k^-(f)})+\chi(\overline{A_{k+2}^+(f)})-\chi(\overline{A_{k+2}^-(f)}),$$ for $k=1, \ldots, n-2$ if $n$ is odd, or $k=1, \ldots, n-3$ if $n$ is even.

By Proposition \ref{PropDF}, if $k$ is odd, then $\overline{A_k^+(f)}$ and $\overline{A_k^-(f)}$ are manifolds with boundary of dimension $n-k$, such that $\partial\overline{A_{k}^+(f)}=\overline{A_{k+1}(f)}=\partial\overline{A_{k}^-(f)}$, $\overline{A_{k}^+(f)}=A_{k}^+(f)\cup\overline{A_{k+1}(f)}$ and $\overline{A_{k}^-(f)}=A_{k}^-(f)\cup\overline{A_{k+1}(f)}$. Furthermore, by Theorem \ref{teobordo} if $M$ is a Riemannian manifold and $F:\overline{A_k^+(f)}\rightarrow\R$ is a correct Morse function, then \begin{equation}\label{chiequation}\chi(\overline{A_k^+(f)})=\displaystyle\sum_{p\in C(F^{\circ})}{(-1)^{\lambda^+_k(p)}}+\displaystyle\sum_{p\in C(\partial F),\nabla F \text{ points inwards}}{(-1)^{\overline{\lambda}_k(p)}},\end{equation} where $F^{\circ}$ denotes the restriction $F|_{A_k^+(f)}$, $\partial F$ denotes the restriction $F|_{\partial\overline{A_k^+(f)}}$, $\lambda^+_k(p)$ denotes the Morse index of $F^{\circ}$ at a critical point $p\in A_k^+(f)$ and $\overline{\lambda}_k(p)$ denotes the Morse index of $\partial F$ at a correct critical point $p$, such that the gradient $\nabla F(p)$ is pointing inwards the manifold $\overline{A_k^+(f)}$. 
Similarly, the result holds for $F:\overline{A_k^-(f)}\rightarrow\R$.\\

From now on, we assume that $M$ is provided with a Riemannian metric. Let us consider the mappings $L_a\circ f$ from Section \ref{lemmasFukuda} and let us suppose that $a\in\R^n\setminus\{\vec{0}\}$ is such that $L_a\circ f(x)$, $L_a\circ f|_{A_k(f)}$ and $L_a\circ f|_{\overline{A_k(f)}}$ are Morse functions that satisfy the properties of Fukuda's lemmas. 
We will apply formula (\ref{chiequation}) to the functions $L_a\circ f|_{\overline{A_k^{+}(f)}}$ and $L_a\circ f|_{\overline{A_k^{-}(f)}}$ in order to study $\chi(\overline{A_k^{+}(f)})$ and $\chi(\overline{A_k^{-}(f)})$, respectively. To do this, it is necessary to make some considerations about correct and non-correct critical points of $L_a\circ f|_{\overline{A_k(f)}}$.

We have to verify if the critical points of the restriction $L_a\circ f|_{\partial\overline{A_k^{+}(f)}}$ are indeed correct critical points. That is, given a point $p\in C(L_a\circ f|_{\overline{A_{k+1}(f)}})$ we have to check that $p\notin C(L_a\circ f|_{\overline{A_k(f)}}).$\\

If $p\in C(L_a\circ f|_{\overline{A_{k+1}(f)}})$, then $p\in\overline{A_{k+1}(f)}\setminus\overline{A_{k+3}(f)}= A_{k+1}(f)\cup A_{k+2}(f)$. Suppose that $p\in A_{k+2}(f)$. By Lemma \ref{Laproperties} item 2, $p\notin C(L_a\circ f|_{\overline{A_k(f)}})$. Thus, $p$ is correct.

On the other hand, if $p\in A_{k+1}(f)\cap C(L_a\circ f|_{\overline{A_{k+1}(f)}})$, then $p\in C(L_a\circ f|_{A_{k+1}(f)})$ and by Lemma \ref{lemaseparado}, $p\in C(L_a\circ f|_{\overline{A_k(f)}})$. Therefore, $p$ is not correct in this case.\\

As we mentioned in Proposition \ref{abertodenso}, the correct Morse functions defined on a compact manifold $N$ form an open and dense set in the function space $C^{\infty}(N,\R)$. 
Indeed, we will prove in the following lemma that it is possible to slightly perturb the function $L_a\circ f$ in neighborhoods of non-correct critical points in order to obtain Morse functions which critical points at the boundary are all corrects. Then, we will apply (\ref{chiequation}) to calculate $\chi(\overline{A_k^+(f)})$ and $\chi(\overline{A_k^-(f)})$.  

\begin{lema}\label{lemadaperturbacao} Let $p\in A_{k+1}(f)$ be a non-correct critical point of $L_a\circ f|_{\overline{A_k(f)}}$. Then:
\begin{enumerate}
\item There exists a perturbation $\widetilde{L_a\circ f}$ of $L_a\circ f|_{\overline{A_k(f)}}$ in a neighborhood $\mathcal{U}_p$ of $p$, such that $\widetilde{L_a\circ f}$ is a correct Morse function and $p\in C(\widetilde{L_a\circ f})$;
\item In the neighborhood $\mathcal{U}_p$, $\exists!$ $\tilde{p}\in C(\widetilde{L_a\circ f})$ such that $\tilde{p}\in A_k^+(f)\cup A_k^-(f)$;
\item The Morse indices of $\widetilde{L_a\circ f}$ at the points $p$ and $\tilde{p}$ are such that $p$ and $\tilde{p}$ do not interfere in the computation of $\chi\left(\overline{A_k^{+}(f)}\right)-\chi\left(\overline{A_k^{-}(f)}\right)$.
\end{enumerate}
\end{lema}

\begin{proof} 

\item[\textit{1.}] Let $p\in A_{k+1}(f)$ be a non-correct critical point of $L_a\circ f|_{\overline{A_k(f)}}$. 
%Since $\overline{A_k(f)}$, $\overline{A_{k}^+(f)}$ and $\overline{A_{k}^-(f)}$ are manifolds $(n-k)$-dimensional and $\partial \overline{A_{k}^+(f)}=\partial\overline{A_{k}^-(f)}=\overline{A_{k+1}(f)}$ is a manifold of dimension $n-k-1$,
We consider local coordinates $x=(x_1,\ldots,x_{n-k-1}, x_{n-k})$ in a neighborhood $\vizp\subset\overline{A_{k}(f)}$ such that $p\in\vizp$, $x(p)=\vec{0}$. In $\vizp$, $L_a\circ f|_{\overline{A_{k+1}(f)}}$ may be written as \begin{equation}\label{Laemakmaisum} L_a\circ f(x_1,\ldots,x_{n-k-1},0)=-x_1^2-\ldots -x_{\lambda}^2+x_{\lambda +1}^2+\ldots+x_{n-k-1}^2\end{equation} and \small \begin{equation}\label{Laemak} L_a\circ f(x_1,\ldots,x_{n-k-1}, x_{n-k})=L_a\circ f(x_1,\ldots,x_{n-k-1},0)+h(x_1,\ldots,x_{n-k}),\end{equation} \normalsize for some function $h:x(\vizp)\subset\R^{n-k}\rightarrow\R$.  Moreover, there exists a small enough $\varepsilon>0$ such that the perturbation \begin{equation}\label{perturbation}\widetilde{L_a\circ f}(x_1,\ldots,x_{n-k})=L_a\circ f(x_1,\ldots,x_{n-k}) + \varepsilon x_{n-k} \end{equation} is also a Morse function, defined in an open neighborhood $\mathcal{V}\subset\R^{n-k}$, with $\vec{0}\in\mathcal{V}$. In particular, this small perturbation coincides with the initial function on $\overline{A_{k+1}(f)}$, next to $p$: $\widetilde{L_a\circ f}(x_1,\ldots,x_{n-k-1},0)=L_a\circ f(x_1,\ldots,x_{n-k-1},0)$; and the Hessian matrices coincide at the origin: $Hess(\widetilde{L_a\circ f}|_{\overline{A_{k+1}(f)}},\vec{0})=Hess(L_a\circ f|_{\overline{A_{k+1}(f)}},\vec{0}).$

Differentiating the equation (\ref{perturbation}), we obtain:
\begin{equation*}\renewcommand{\arraystretch}{2.5}{
\ffrac{\partial(\widetilde{L_a\circ f})}{\partial x_i}(\vec{0})=\left\{\begin{array}{ll}
\ffrac{\partial(L_a\circ f)}{\partial x_i}(\vec{0})=0, & \ i=1,\ldots, n-k-1;\\
\ffrac{\partial(L_a\circ f)}{\partial x_{n-k}}(\vec{0})+\varepsilon=\varepsilon, & \ i=n-k.
\end{array}\right.}
\end{equation*} Thus, the gradient of the function $\widetilde{L_a\circ f}$ restricted to the boundary $\overline{A_{k+1}(f)}$ at $\vec{0}\in\R^{n-k}$ is zero, $\nabla(\widetilde{L_a\circ f}|_{\overline{A_{k+1}}})(\vec{0})=(0,\ldots,0)$, and the gradient of the function $\widetilde{L_a\circ f}$ at $\vec{0}\in\R^{n-k}$, $\nabla(\widetilde{L_a\circ f})(\vec{0})=(0,\ldots,0,\varepsilon)$, is not zero. That is, $p$ is a correct critical point of $\widetilde{L_a\circ f}|_{\overline{A_k^+(f)}}$ and $\widetilde{L_a\circ f}|_{\overline{A_k^-(f)}}$. \\

\item[\textit{2.}] Since $p$ is a non-degenerate critical point of $L_a\circ f|_{\overline{A_k(f)}}$ then we have $\nabla(L_a\circ f)(\vec{0})=\vec{0}$ and $\det[Hess({L_a\circ f})(\vec{0})]\neq0.$ However, $Hess({L_a\circ f})(\vec{0})=Jac[\nabla({L_a\circ f})(\vec{0})].$ Thus, the map $\nabla(L_a\circ f):\mathcal{V}\subset\R^{n-k}\rightarrow\R^{n-k}$ is a local diffeomorphism in $\vec{0}\in\R^{n-k}$. Thereby, $\exists!$ $\tilde{p}$ in a neighborhood of $\vec{0}$ such that $\nabla(L_a\circ f)(\tilde{p})=(0,\ldots,0,-\varepsilon)$ and, consequently, $\exists!$ $\tilde{p}$ in a neighborhood of $\vec{0}$ such that $\nabla(\widetilde{L_a\circ f})(\tilde{p})=(0,\ldots,0).$ That is, $\tilde{p}\in C(\widetilde{L_a\circ f})\setminus C(L_a\circ f|_{\overline{A_k(f)}})$. In particular, we may consider such neighborhood such that \begin{equation}\label{hesscomparacao}\s\det\left[ Hess(L_a\circ f)(\vec{0})\right]=\s\det\left[ Hess(\widetilde{L_a\circ f})(\tilde{p})\right],\end{equation} where $\s$ denotes the sign of the determinants. 

In the next step, we will show that $x_{n-k}(\tilde{p})\neq0$. Let $x=(x_1,\ldots,x_{n-k})$, since $\nabla(L_a\circ f)(\tilde{p})=(0,\ldots,0,-\varepsilon)$, then $\tilde{p}$ is the solution of the system of linear equations:
\begin{equation}\label{sistema}
\left\{\renewcommand{\arraystretch}{2}{\begin{array}{ll}
\ffrac{\partial (L_a\circ f)}{\partial x_j}(x)=0, & j=1,\ldots, n-k-1;\\
\ffrac{\partial (L_a\circ f)}{\partial x_{n-k}}(x)=-\varepsilon.
\end{array}}\right.
\end{equation}
%%%%%%%%%%%%%%%%%%%%%%%%%%%%%%%%%%%%%%%%%%%%%%%%%%%%%%%%%%%%%%%%%%%%%%%%%%%%%%%%%%%%%%%%%%%%%%%%%%%%%%%%%%%%%%%%%%%%%%%%%%%%%%%%%%%%
By equations (\ref{Laemakmaisum}) and (\ref{Laemak}), we have \begin{equation}\label{derivparciais}\ffrac{\partial(L_a\circ f)}{\partial x_j}(x)=\left\{\renewcommand{\arraystretch}{2.3}{\begin{array}{cl}
-2x_j+\ffrac{\partial h}{\partial x_j}(x), & \ j=1,\ldots, \lambda;\\
2x_j+\ffrac{\partial h}{\partial x_j}(x), & \ j=\lambda+1,\ldots,n-k-1;\\
\ffrac{\partial h}{\partial x_{n-k}}(x), & \ j=n-k.
\end{array}}\right.
\end{equation} Clearly, by equation (\ref{Laemak}), $h(x_1,\ldots, x_{n-k-1},0)=0, \, \forall (x_1,\ldots, x_{n-k-1},0)\in\R^{n-k}.$
Hence $h(\vec{0})=0$ and, since $\vec{0}$ is a critical point of $L_a\circ f|_{\overline{A_{k}(f)}}$, equation (\ref{derivparciais}) implies that $\ffrac{\partial h}{\partial x_i}(\vec{0})$ vanishes for all $i=1,\ldots,n-k.$

Since $h(\vec{0})=0$, by Hadamard's Lemma, there exist smooth functions $h_i$, $i=1,\ldots, n-k$, defined in an open neighborhood of $\vec{0}$ in $\R^{n-k}$, such that $h(x)=\sum_{i=1}^{n-k}{x_ih_i(x)}$, for $j=1,\ldots, n-k$. It follows that, %for $j=1,\ldots, n-k,$ 
\begin{equation}\label{partialh}\begin{array}{c}
\ffrac{\partial h}{\partial x_j}(x)=h_j(x)+\displaystyle\sum_{i=1}^{n-k}{x_i\ffrac{\partial h_i}{\partial x_j}(x)}\Rightarrow\ffrac{\partial h}{\partial x_j}(\vec{0})=h_j(\vec{0}), \ j=1,\ldots, n-k. 
\end{array}
\end{equation} Hence, $h_j(\vec{0})=0$ for all $j=1,\ldots, n-k$ and also by Hadamard's Lemma, there exist smooth functions $l_{i,j}$, $i,j=1, \ldots, n-k$, defined in a open neighborhood of $\vec{0}$ in $\R^{n-k}$, such that \begin{equation}\label{hj} h_j(x)=\displaystyle\sum^{n-k}_{i=1}{x_il_{i,j}(x)}.\end{equation} Replacing relations (\ref{partialh}) and (\ref{hj}) in (\ref{derivparciais}), the partial derivatives become
\small
\begin{equation}\label{parciais}\ffrac{\partial(L_a\circ f)}{\partial x_j}(x)=\left\{\setlength{\arraycolsep}{0.1cm}{\renewcommand{\arraystretch}{2.5}{\begin{array}{rl}
-2x_j+ \displaystyle\sum_{i=1}^{n-k}{x_i\left(l_{i,j}(x)+\ffrac{\partial h_i}{\partial x_j}(x)\right)}, & j=1,\ldots, \lambda;\\
2x_j+\displaystyle\sum_{i=1}^{n-k}{x_i\left(l_{i,j}(x)+\ffrac{\partial h_i}{\partial x_j}(x)\right)}, & j=\lambda+1,\ldots,n-k-1;\\
\displaystyle\sum_{i=1}^{n-k}{x_i\left(l_{i,n-k}(x)+\ffrac{\partial h_i}{\partial x_{n-k}}(x)\right)}, & j=n-k.
\end{array}}}\right.
\end{equation}\normalsize 

Let $A(x)=\left(a_{i,j}(x)\right)$ be the matrix whose entries $a_{i,j}(x)$ are given by the coefficients of the monomials $x_j$ in the expression $\ffrac{\partial (L_a\circ f)}{\partial x_i}(x)$, for $i,j=1, \ldots, n-k$, that is: \begin{equation}\label{coeficientes}\ffrac{\partial (L_a\circ f)}{\partial x_i}(x)=a_{i,1}(x)x_1+\ldots+a_{i,n-k}(x)x_{n-k}, \ i=1, \ldots, n-k.\end{equation} Therefore, by equations (\ref{parciais}) and (\ref{coeficientes}), the coefficients $a_{i,j}(x)$ are given by:
\small \begin{equation}
a_{i,j}(x)=\left\{\setlength{\arraycolsep}{0.15cm}{\renewcommand{\arraystretch}{2.5}{\begin{array}{llll}
-2+l_{i,i}(x)&+&\ffrac{\partial h_i}{\partial x_i}(x), & i,j=1,\ldots,\lambda \text{ and } i=j;\\
2+l_{i,i}(x)&+&\ffrac{\partial h_i}{\partial x_i}(x), & i,j=\lambda+1,\ldots,n-k-1 \text{ and } i=j;\\
l_{j,i}(x)&+&\ffrac{\partial h_j}{\partial x_i}(x), & i,j=1,\ldots,n-k \text{ and } i\neq j;\\
l_{n-k,n-k}(x)&+&\ffrac{\partial h_{n-k}}{\partial x_{n-k}}(x), & i=j=n-k,
\end{array}}}\right.
\end{equation} \normalsize and the system of linear equations (\ref{sistema}) can be expressed in the matrix form \begin{equation}\label{sistemaA} 
A(x)\left(\begin{array}{c}
x_1\\
\vdots\\
x_{n-k-1}\\
x_{n-k}
\end{array}\right)=\left(\begin{array}{c}
0\\
\vdots\\
0\\
-\varepsilon
\end{array}\right).
\end{equation}

Let $A_{n-k}(x)$ be the matrix obtained by replacing the last column of $A(x)$ by the vector $(0,\ldots,0,-\varepsilon)$. Since $\tilde{p}$ is a solution of the system (\ref{sistemaA}), if $\det\left[A(\tilde{p})\right]\neq0$ then \begin{equation}\label{sinalxnk}
x_{n-k}(\tilde{p})= \ffrac{\det\left[A_{n-k}(\tilde{p})\right]}{\det\left[A(\tilde{p})\right]}.
\end{equation} Thus, if the relation (\ref{sinalxnk}) is true and $\det\left[A_{n-k}(\tilde{p})\right]\neq0$, then we obtain 
\begin{equation}
\s x_{n-k}(\tilde{p})=\s\det\left[A_{n-k}(\tilde{p})\right] \s\det\left[A(\tilde{p})\right].
\end{equation} In order to verify that $\det\left[A(\tilde{p})\right]\neq0$, first we calculate the matrix $A(\vec{0})$.

By (\ref{hj}), for $i,j=1, \ldots,n-k$, we have 
\begin{equation}
\ffrac{\partial h_j}{\partial x_i}(x)=l_{i,j}(x)+\displaystyle\sum_{r=1}^{n-k}x_r\ffrac{\partial l_{r,j}}{\partial x_i}(x),
\end{equation} so that $\ffrac{\partial h_j}{\partial x_i}(\vec{0})=l_{i,j}(\vec{0})$, for all $i,j=1,\ldots, n-k$. Then, evaluating the matrix $A(x)=(a_{i,j}(x))$ at $\vec{0}\in\R^{n-k}$, we obtain
\begin{equation}\label{aijzero}
a_{i,j}(\vec{0})=\left\{\renewcommand{\arraystretch}{2.5}{\begin{array}{ll}
-2+2\ffrac{\partial h_i}{\partial x_i}(\vec{0}), & \ i,j=1,\ldots,\lambda \text{ and } i=j;\\
2+2\ffrac{\partial h_i}{\partial x_i}(\vec{0}), & \ i,j=\lambda+1,\ldots,n-k-1 \text{ and } i=j;\\
\ffrac{\partial h_i}{\partial x_j}(\vec{0})+\ffrac{\partial h_j}{\partial x_i}(\vec{0}), & \ i,j=1,\ldots,n-k \text{ and } i\neq j;\\
2\ffrac{\partial h_{n-k}}{\partial x_{n-k}}(\vec{0}), & \ i=j=n-k.
\end{array}}\right.
\end{equation}

%Deriving once more $L_a\circ f(x)$ in (\ref{derivparciais}), deriving $\ffrac{\partial h}{\partial x_j}(x)$ in (\ref{partialh}), for $i,j=1,\ldots,n-k$ and comparing these partial derivatives of order two evaluated at $\vec{0}\in\R^{n-k}$, we obtain $a_{i,j}(\vec{0})=\ffrac{\partial^2(L_a\circ f)}{\partial x_i\partial x_j}(\vec{0})$, $\forall i,j=1, \ldots,n-k$, that is, \begin{equation}\label{Azero}Hess(L_a\circ f)(\vec{0})=A(\vec{0}).\end{equation}

%%%%%%%%%%%%%%%%%%%%%%%%%%%%%%%%%%%%%%%%%%%%%%%%%%%%%%%%%%%%%%%%%%%%%%%%%%%%%%%%%%%%%%%%%%%%%%%%%%%%%%%%%%%%%%%%%%%%%%%%%%%%%%%%%%%%%

Let us show that $A(\vec{0})=Hess(L_a\circ f|_{\overline{A_k(f)}},\vec{0})$. Differentiating once more $L_a\circ f(x)$ in (\ref{derivparciais}), we obtain: for $i,j=1,\ldots, \lambda$,
\begin{equation*}\ffrac{\partial^2(L_a\circ f)}{\partial x_i\partial x_j}(x)=\left\{\setlength{\arraycolsep}{0.12cm}{\renewcommand{\arraystretch}{2.5}{\begin{array}{ll}
-2+\ffrac{\partial^2 h}{\partial x_j^2}(x), & i=j;\\
\ffrac{\partial^2 h}{\partial x_i\partial x_j}(x), & i\neq j;
\end{array}}}\right.
\end{equation*} 
for $i,j= \lambda+1,\ldots,n-k-1$, \begin{equation*}\ffrac{\partial^2(L_a\circ f)}{\partial x_i\partial x_j}(x)=\left\{\setlength{\arraycolsep}{0.12cm}{\renewcommand{\arraystretch}{2.5}{\begin{array}{ll}
2+\ffrac{\partial^2 h}{\partial x_j^2}(x), & i=j;\\
\ffrac{\partial^2 h}{\partial x_i\partial x_j}(x), & i\neq j;
\end{array}}}\right.
\end{equation*} 
and for $j=1,\ldots,\lambda$ and $i= \lambda+1,\ldots,n-k,$ $$\ffrac{\partial^2(L_a\circ f)}{\partial x_i\partial x_j}(x)=\ffrac{\partial^2 h}{\partial x_i\partial x_j}(x).$$ 
Finally, for $i= \lambda+1,\ldots,n-k$, $$\ffrac{\partial^2(L_a\circ f)}{\partial x_i\partial x_{n-k}}(x)=\ffrac{\partial^2 h}{\partial x_i\partial x_{n-k}}(x).$$

Differentiating $\ffrac{\partial h}{\partial x_j}(x)$ in (\ref{partialh}), for $i,j=1,\ldots,n-k$, we have
\begin{equation*}\setlength{\arraycolsep}{0.12cm}{\renewcommand{\arraystretch}{2.5}{\begin{array}{ccl}
\ffrac{\partial^2 h}{\partial x_j^2}(x)&=&2\ffrac{\partial h_j}{\partial x_j}(x)+\displaystyle\sum_{r=1}^{n-k}x_r\ffrac{\partial^2 h_r}{\partial x_j^2}(x);\\
\ffrac{\partial^2 h}{\partial x_i\partial x_j}(x)&=&\ffrac{\partial h_j}{\partial x_i}(x)+\ffrac{\partial h_i}{\partial x_j}(x)+\displaystyle\sum_{r=1}^{n-k}x_r\ffrac{\partial^2 h_r}{\partial x_i\partial x_j}(x).
\end{array}}}
\end{equation*} Then, \begin{equation*}\setlength{\arraycolsep}{0.12cm}{\renewcommand{\arraystretch}{2.5}{\begin{array}{ccl}
\ffrac{\partial^2 h}{\partial x_j^2}(\vec{0})&=&2\ffrac{\partial h_j}{\partial x_j}(\vec{0}),\\
\ffrac{\partial^2 h}{\partial x_i\partial x_j}(\vec{0})&=&\ffrac{\partial h_j}{\partial x_i}(\vec{0})+\ffrac{\partial h_i}{\partial x_j}(\vec{0}).
\end{array}}}
\end{equation*}

Evaluating $\ffrac{\partial^2(L_a\circ f)}{\partial x_i\partial x_j}(x)$ at $\vec{0}\in\R^{n-k}$, we verify that: for $i,j=1,\ldots,\lambda$,\begin{equation*}\ffrac{\partial^2(L_a\circ f)}{\partial x_i\partial x_j}(\vec{0})=\left\{\setlength{\arraycolsep}{0.12cm}{\renewcommand{\arraystretch}{2.5}{\begin{array}{ll}
-2+2\ffrac{\partial h_j}{\partial x_j}(\vec{0}), & i=j;\\
\ffrac{\partial h_j}{\partial x_i}(\vec{0})+\ffrac{\partial h_i}{\partial x_j}(\vec{0}), & i\neq j;
\end{array}}}\right.
\end{equation*} 
for $i,j=\lambda+1,\ldots,n-k-1$, \begin{equation*}\ffrac{\partial^2(L_a\circ f)}{\partial x_i\partial x_j}(\vec{0})=\left\{\setlength{\arraycolsep}{0.12cm}{\renewcommand{\arraystretch}{2.5}{\begin{array}{ll}
2+2\ffrac{\partial h_j}{\partial x_j}(\vec{0}), & i=j;\\
\ffrac{\partial h_j}{\partial x_i}(\vec{0})+\ffrac{\partial h_i}{\partial x_j}(\vec{0}), & i\neq j;
\end{array}}}\right.
\end{equation*} 
and for $j=1,\ldots,\lambda$ and $i=\lambda+1,\ldots,n-k$, $$\ffrac{\partial^2(L_a\circ f)}{\partial x_i\partial x_{i}}(\vec{0})=\ffrac{\partial h_j}{\partial x_i}(\vec{0})+\ffrac{\partial h_i}{\partial x_j}(\vec{0}).$$ 
Finally, for $i= \lambda+1, \ldots, n-k$ $$\ffrac{\partial^2(L_a\circ f)}{\partial x_i \partial x_{n-k}}(\vec{0})=\ffrac{\partial h_{n-k}}{\partial x_i}(\vec{0})+\ffrac{\partial h_i}{\partial x_{n-k}}(\vec{0}).$$

Therefore, comparing these second-order derivatives with coefficients (\ref{aijzero}) we conclude that $a_{i,j}(\vec{0})=\ffrac{\partial^2(L_a\circ f)}{\partial x_i\partial x_j}(\vec{0})$ for all $i,j=1, \ldots,n-k$, that is, $Hess(L_a\circ f)(\vec{0})=A(\vec{0}).$

%%%%%%%%%%%%%%%%%%%%%%%%%%%%%%%%%%%%%%%%%%%%%%%%%%%%%%%%%%%%%%%%%%%%%%%%%%%%%%%%%%%%%%%%%%%%%%%%%%%%%%%%%%%%%%%%%%%%%%%%%%%%%%%%%%%%
Remember that $p$ is a non-degenerate critical point of $L_a\circ f|_{\overline{A_k(f)}}$, then we have that $\det[Hess(L_a\circ f)(\vec{0})]\neq0.$ Therefore, $\det[A(\vec{0})]\neq0$.

Let us consider the neighborhood of $\vec{0}\in\R^{n-k}$ small enough such that \begin{equation}\label{sinalAzero}\s\det[A(\tilde{p})]= \s\det[A(\vec{0})].\end{equation}

%The functions that determine the matrix $A(x)$ are continuous, thus, we may consider the neighborhood of $\vec{0}$ in $\R^{n-k}$, in which we are working, enough small, such that \begin{equation}\label{sinalAzero}\s \det A(\tilde{p})= \s\det A(\vec{0}).\end{equation}

Since $\overline{A_k(f)}=\overline{A_k^+(f)}\cup\overline{A_k^-(f)}$ and $\partial\overline{A_k^+(f)}=\partial\overline{A_k^-(f)}=\partial\overline{A_{k+1}(f)}$, we may take $x_1, \ldots, x_{n-k-1}$ as local coordinates of $\overline{A_{k+1}(f)}$, such that the Hessian matrix of the function $L_a\circ f|_{\overline{A_{k+1}(f)}}$ at $\vec{0}$ is the submatrix of $Hess(L_a\circ f)(\vec{0})$ obtained by removing the last line and the last column of $Hess(L_a\circ f)(\vec{0})$. In particular, the Hessian matrix $Hess(L_a\circ f|_{\overline{A_{k+1}(f)}})(\vec{0})$ is also non-singular, because $p$ is a non-degenerate critical point of $L_a\circ f|_{\overline{A_{k+1}(f)}}$ (see Lemma \ref{Laproperties}, item 3). Then, \begin{equation}\label{hessakmaisum}\det[ A_{n-k}(\vec{0})]=-\varepsilon\det[Hess(L_a\circ f|_{\overline{A_{k+1}}})(\vec{0})]\neq0,\end{equation} and again, we may assume that \begin{equation}\label{sinalAnmenosk}\s\det[A_{n-k}(\tilde{p})]= \s\det[A_{n-k}(0)].\end{equation}

Thereby, equation (\ref{sinalxnk}) implies that $x_{n-k}(\tilde{p})\neq0$. Therefore $\tilde{p}\notin\overline{A_{k+1}(f)}$, that is, $\tilde{p}\in A_{k}^+(f)\cup A_{k}^-(f)$. 

\item[\textit{3.}] We want to verify under which conditions $\tilde{p}\in A_k^+(f)$ or $\tilde{p}\in A_k^-(f)$ and how it interferes in the calculus of $\chi\left(\overline{A_k^{+}(f)}\right)-\chi\left(\overline{A_k^{-}(f)}\right)$. To do this, it is enough to study the sign of $x_{n-k}(\tilde{p})$. By previous item's proof, we have \begin{equation}\label{sinal}\renewcommand{\arraystretch}{1.5}{\begin{array}{llll}
\s x_{n-k}(\tilde{p})&=& &\s\det[A_{n-k}(\vec{0})].\s\det[A(\vec{0})]\\
&=&-&\s\det[Hess(L_a\circ f|_{\overline{A_{k+1}(f)}})(\vec{0})].\s\det[Hess(L_a\circ f)(\vec{0})]\\
&=&-&\s\det[Hess(\widetilde{L_a\circ f}|_{\overline{A_{k+1}(f)}})(\vec{0})].\s\det[Hess(\widetilde{L_a\circ f})(\tilde{p})]\\
&=&&(-1)^{Ind(\widetilde{L_a\circ f}|_{A_{k+1}(f)},\vec{0})+1}(-1)^{Ind(\widetilde{L_a\circ f},\tilde{p})}
\end{array}}
\end{equation} 

We have seen that $\nabla(\widetilde{L_a\circ f})(\vec{0})=(0,\ldots,0,\varepsilon)$ which is equal to $\varepsilon\nabla\{x_{n-k}\}(\vec{0})$, with $\varepsilon>0.$ Thus, the vector $\nabla(\widetilde{L_a\circ f})(\vec{0})$ is pointing inwards in the half-space determined by $x_{n-k}>0$. Since $p\in A_{k+1}(f)$, there exist local coordinates $(u_1,\ldots, u_m)$ around $p$ in $M$ and local coordinates $(z_1,\ldots, z_n)$ around $f(p)$ in $\R^n$ such that, locally, $f$ has the normal form:  \begin{equation}\label{formalocalak+1}
\begin{array}{l}
 z_i\circ f=u_i \text{, for } i\leq n-1,\\
 z_n\circ f=u_n^{k+2}+\displaystyle\sum_{i=1}^{k}u_iu_n^{k+1-i}+ u_{n+1}^2+\ldots+ u_{n+\mu-1}^2-u_{n+\mu}^2-\ldots-u_m^2.
\end{array}
\end{equation} By Morin's characterization of $A_k$ singularities (see \cite{Morin2}, \cite[p. 342]{Fukuda}), in the considered neighborhood, we may write \small
\begin{equation*}\renewcommand{\arraystretch}{2.7}{\setlength{\arraycolsep}{0.02cm}{\begin{array}{cl}
A_k(f)&=\left\{\ffrac{\partial^{j}(z_n\circ f)}{\partial u_n^j}=0, j=1,\ldots,k;\ffrac{\partial(z_n\circ f)}{\partial u_{i}}=0, i=n+1,\ldots,m; \ffrac{\partial^{k+1}(z_n\circ f)}{\partial u_n^{k+1}}\neq0\right\};\\
\overline{A_{k}(f)}&=\left\{\ffrac{\partial^{j}(z_n\circ f)}{\partial u_n^j}=0, j=1, \ldots,k+1;\ffrac{\partial (z_n\circ f)}{\partial u_{i}}=0, i=n+1,\ldots,m; \right\};
\end{array}}}
\end{equation*} \normalsize so that $$\begin{array}{ccl}
A_k(f)&=&\left\{u_r-c_ru_n^{r+1}=0, r=1,...,k; u_{n+1}=\ldots=u_m=0; u_n\neq0\right\};\\
\overline{A_k(f)}&=&\left\{u_r-c_ru_n^{r+1}=0, r=1,...,k; u_{n+1}=\ldots=u_m=0;\right\};\\
\end{array}$$ where $c_r\in\R$ are constants. Thus, by the Implicit Function Theorem, $u_{k+1}, \ldots, u_n$ are local coordinates in $\overline{A_k(f)}$.

If $k$ and $m-n>0$ are odd, in \cite[Proposition 6.1, p. 188]{Dutertrefukui}, N. Dutertre and T. Fukui have shown a characterization for the manifolds $A_{k}^+(f)$ and $A_{k}^-(f)$ given in terms of the parity of the Morse index $\mu$ of the quadratic part of $z_n\circ f$. According to this characterization, we consider the following cases:

\begin{itemize}
\item If $\mu$ is even, then: $q\in A_{k}^+(f)\Leftrightarrow u_{n}(q)>0$ and $q\in A_{k}^-(f)\Leftrightarrow u_{n}(q)<0$;

%\begin{equation*}\begin{array}{lll}
%q\in A_{k}^+(f)&\Leftrightarrow &u_{n}(q)>0;\\
%q\in A_{k}^-(f)&\Leftrightarrow &u_{n}(q)<0;\\
%\end{array}
%\end{equation*}
\item If $\mu$ is even, then: $q\in A_{k}^+(f)\Leftrightarrow u_{n}(q)<0$ and $q\in A_{k}^-(f)\Leftrightarrow u_{n}(q)>0$;
%\begin{equation*}\begin{array}{lll}
%q\in A_{k}^+(f)&\Leftrightarrow &u_{n}(q)<0;\\
%q\in A_{k}^-(f)&\Leftrightarrow &u_{n}(q)>0;\\
%\end{array}
%\end{equation*}
\end{itemize} where $q=(u_1(q),...,u_{m}(q))$ is a point near $p$.
Note that we have two local coordinates systems of $\overline{A_k(f)}$ around $p$ given by $(x_1, \ldots, x_{n-k})$ and $(u_{k+1}, \ldots, u_n)$, such that the hyperplanes $x_{n-k}=0$ and $u_n=0$ are equal. Let us suppose, without loss of generality, that the half-space determined by $x_{n-k}>0$ corresponds to the half-space determined by $u_n>0$. Then, we have:
\begin{itemize}
\item If $\mu$ is even, then: $q\in A_{k}^+(f)\Leftrightarrow x_{n-k}(q)>0$ and $q\in A_{k}^-(f)\Leftrightarrow x_{n-k}(q)<0;$
\item If $\mu$ is odd, then: $q\in A_{k}^+(f)\Leftrightarrow x_{n-k}(q)<0$ and $q\in A_{k}^-(f)\Leftrightarrow x_{n-k}(q)>0;$
\end{itemize} where $q=(x_1(q),...,x_{m}(q))$ is a point near $p$.\\

Finally, let us see that the critical points $p$ and $\tilde{p}$ do not interfere in the calculus of $\chi\left(\overline{A_k^{+}(f)}\right)-\chi\left(\overline{A_k^{-}(f)}\right)$.\\ 

If $\mu$ is even, then $q\in A_{k}^+(f)\Leftrightarrow x_{n-k}(q)>0$, that is, $\{x_{n-k}>0\}=A_k^+(f).$ Then, $\nabla(\widetilde{L_a\circ f})(\vec{0})$ is pointing inwards $A_k^+(f)$. 
\begin{itemize}
\item Suppose that $\tilde{p}\in A_k^+(f)$. Then $x_{n-k}(\tilde{p})>0$ and, by equation (\ref{chiequation}), the contribution of $p$ and $\tilde{p}$ in the computation of $\small{\chi(\overline{A_k^{+}(f)})-\chi(\overline{A_k^{-}(f)})}$ is $$(-1)^{Ind(\widetilde{L_a\circ f},\tilde{p})}+(-1)^{Ind(\widetilde{L_a\circ f}|_{A_{k+1}(f)},\vec{0})}=0$$ because equation (\ref{sinal}) tell us that the respective Morse indices have opposite parities, since $x_{n-k}(\tilde{p})>0$.  
\item Suppose that $\tilde{p}\in A_k^-(f)$. Then, similarly, $x_{n-k}(\tilde{p})<0$ and, by equation (\ref{chiequation}), the contribution of $p$ and $\tilde{p}$ in the computation of $\small{\chi(\overline{A_k^{+}(f)})-\chi(\overline{A_k^{-}(f)})}$ is $$(-1)^{Ind(\widetilde{L_a\circ f}|_{A_{k+1}(f)},\vec{0})}-(-1)^{Ind(\widetilde{L_a\circ f},\tilde{p})}=0.$$
\end{itemize}

If $\mu$ is odd, then $q\in A_{k}^+(f)\Leftrightarrow x_{n-k}(q)<0$, that is, $\{x_{n-k}>0\}=A_k^-(f).$ Then, $\nabla(\widetilde{L_a\circ f})(\vec{0})$ is pointing inwards $A_k^-(f)$ and, with the argument of the previous computation, we have that the contribution of $p$ and $\tilde{p}$ in the computation of $\small{\chi(\overline{A_k^{+}(f)})-\chi(\overline{A_k^{-}(f)})}$ is zero whether $\tilde{p}\in A_k^+(f)$ or $\tilde{p}\in A_k^-(f)$.

%\begin{itemize}
%\item Suppose that $\tilde{p}\in A_k^+(f)$. Then $x_{n-k}(\tilde{p})<0$ and, locally, $$\chi\left(A_k^{+}(f)\right)-\chi\left(A_k^{-}(f)\right)=(-1)^{Ind(L_a\circ f,\tilde{p})}-(-1)^{Ind(L_a\circ f|_{A_{k+1}(f)},\vec{0})}=0.$$ %because the expression that determines the sign of $x_{n-k}(\tilde{p})$ tell us that the respective indices have equal parities, since that $x_{n-k}(\tilde{p})<0$.  
%\item Suppose that $\tilde{p}\in A_k^-(f)$. Then $x_{n-k}(\tilde{p})>0$ and, locally, $$\chi\left(A_k^{+}(f)\right)-\chi\left(A_k^{-}(f)\right)=-(-1)^{Ind(L_a\circ f|_{A_{k+1}(f)},\vec{0})}-(-1)^{Ind(L_a\circ f,\tilde{p})}=0.$$ %because the expression that determines the sign of $x_{n-k}(\tilde{p})$ tell us that the respective indices have reverse parities, since that $x_{n-k}(\tilde{p})>0$.  
%\end{itemize}

Therefore the critical points $p\in A_{k+1}(f)$ and $\tilde{p}\in A_k^+(f)\cup A_k^-(f)$ do not interfere in the global computation of $\chi\left(\overline{A_k^{+}(f)}\right)-\chi\left(\overline{A_k^{-}(f)}\right)$, since the indices of the perturbation $\widetilde{L_a\circ f}$ in this singularities cancel each other out.\end{proof}

%%%%%%%%%%%%%%%%%%%%%%%%%%%%%%%%%%%%%%%%%%%%%%%%%%%%%%%% New proof %%%%%%%%%%%%%%%%%%%%%%%%%%%%%%%%%%%%%%%%%%%%%%%%%%%%%%%%%%%%%%%%%%%

\section{A new proof of Dutertre-Fukui's theorem}\label{newproof}

Let us state the following notation around each singular point of $f$. Since $f$ is a Morin map, we know that around a singular point $p\in M$ of type $A_k$, $f$ has the form given by (\ref{formalocalcap2}). We denote by $\psi=(\psi_1,\ldots,\psi_n)$ the inverse mapping of the diffeomorphism $y$, that is $\psi:=y^{-1}$, and by $g=(g_1,\ldots,g_n)$ the composite mapping $y\circ f\circ x^{-1}$. Hence, the coordinate functions of $g$ are given by $g_j(x)=y_j\circ f(x)$, for $j=1,\ldots,n$ and $L_a\circ f(x)$ can be described as \begin{equation*}
L_a\circ f(x)=\displaystyle\sum_{i=1}^{n}{a_if_i(x)}=\displaystyle\sum_{i=1}^{n}{a_i(\psi_i\circ g)(x)}.
\end{equation*}

\begin{teo}\cite[Theorem 6.2, p. 188]{Dutertrefukui}
Let $f:M\rightarrow\R^n$ be a Morin map defined on a compact manifold $m$-dimensional $M$, with $m-n>0$ odd. Then, $$\chi(M)=\sum_{k: \text{ odd}}{\left[\chi(\overline{A_k^+(f)})-\chi(\overline{A_k^-(f)})\right]}.$$
\end{teo}
\begin{proof} Since $M$ is compact and $L_a\circ f$ is a Morse function, it is a well known fact from Morse theory that $$\chi(M)=\sum_{p\in C(L_a\circ f)}(-1)^{\lambda(p)},$$ where $\lambda(p)=Ind(L_a\circ f,p)$ denotes the Morse index of $L_a\circ f$ at a critical point $p$. It is not difficult to see that $C(L_a\circ f)\subset \overline{A_1(f)}$ and by Lemma \ref{Laproperties} item 2, we have that $C(L_a\circ f)\cap\overline{A_2(f)}=\emptyset$. Hence, if $p\in C(L_a\circ f)$, then $p\in A_1(f)= A_1^+(f)\cup A_1^-(f)$ and Lemma \ref{indexlemma} states that \begin{equation}\label{index} Ind(L_a\circ f,p)\equiv\left\{\begin{array}{llll}
 Ind(L_a\circ f|_{A_1^+(f)},p) & \mod2, & \text{ if } p\in A_1^+(f),\\
 1 + Ind(L_a\circ f|_{A_1^-(f)},p) & \mod2, & \text{ if } p\in A_1^-(f).
\end{array}\right.\end{equation} Let us denote $Ind(L_a\circ f|_{A_1(f)},p)$ by $\lambda^1(p)$ and $C(L_a\circ f)$ by $C$, then \begin{equation}\label{ladoesquerdo}\renewcommand{\arraystretch}{2}{
\begin{array}{ccl}
\chi(M)=\displaystyle\sum_{p\in C}(-1)^{\lambda(p)}&=&\displaystyle\sum_{p\in A_1^+(f)\cap C}(-1)^{\lambda(p)}+\displaystyle\sum_{p\in  A_1^-(f)\cap C}(-1)^{\lambda(p)}\\
&=&\displaystyle\sum_{p\in A_1^+(f)\cap C}(-1)^{\lambda^1(p)}+\displaystyle\sum_{p\in A_1^-(f)\cap C}(-1)^{\lambda^1(p)+1}\\
&=&\displaystyle\sum_{p\in A_1^+(f)\cap C}(-1)^{\lambda^1(p)}-\displaystyle\sum_{p\in A_1^-(f)\cap C}(-1)^{\lambda^1(p)}.\\
\end{array}}
\end{equation}

Regarding the sum $$\displaystyle\sum_{k: \text{ odd}}{\left[\chi(\overline{A_k^+(f)})-\chi(\overline{A_k^-(f)})\right]},$$ we have seen in the previous section that the critical points in $C(L_a\circ f|_{\overline{A_{k+1}(f)}})\cap{A_{k+1}(f)}$ are not correct critical points of $L_a\circ f|_{\overline{A_{k}(f)}}$. Hence, we consider a partition of unity composed by $L_a\circ f|_{\overline{A_k(f)}}$ and by the perturbations $\widetilde{L_a\circ f}$ defined in the neighborhoods of the non-correct critical points, in order to obtain a correct Morse function defined on $M$ and to apply Theorem \ref{teobordo}.

For clearer notations, let us consider $C(k):=C(L_a\circ f|_{A_{k}(f)})$, $C(\overline{k}):=C(L_a\circ f|_{\overline{A_{k}(f)}})$, $C(k^+):=C(L_a\circ f|_{A_{k}^+(f)})$ and $C(k^-):=C(L_a\circ f|_{A_{k}^-(f)})$. Moreover, we will denote
\begin{equation}\label{notacaogradientes}\setlength{\arraycolsep}{0.1cm}{\begin{array}{ccl}
\nabla(\overline{k})(p)^{\wedge}&:=&\nabla(L_a\circ f|_{\overline{A_{k}(f)}})(p)\text{ ``is pointing outwards'' };\\
\nabla(\overline{k})(p)^{\vee}&:=&\nabla(L_a\circ f|_{\overline{A_{k}(f)}})(p)\text{ ``is pointing inwards'' }.\\
\end{array}}
\end{equation} 
For instance, $\nabla(\overline{k})(p)^{\wedge} A_{k}^+(f)$ means that the gradient vector of the map $L_a\circ f|_{\overline{A_{k}(f)}}$ at the point $p$ is pointing outwards the manifold $A_{k}^+(f)$. 

By Theorem \ref{teobordo} and Lemma \ref{lemadaperturbacao}, we have
\begin{equation*}\label{chik}\renewcommand{\arraystretch}{1.8}{\setlength{\arraycolsep}{0.1cm}{\begin{array}{cclcl}
\chi(\overline{A_{k}^+(f)})&=&\displaystyle\sum_{p\in C(k^+)}{(-1)^{\lambda^k(p)}}&+&\displaystyle\sum_{\scriptsize\begin{array}{c}p\in C(\overline{k+1})\cap A_{k+2}(f),\\ \nabla(\overline{k})(p)^{\vee} A_{k}^+(f)\end{array}\normalsize}{(-1)^{\overline{\lambda}^{k+1}(p)}};\\
\chi(\overline{A_{k}^-(f)})&=&\displaystyle\sum_{p\in C(k^-)}{(-1)^{\lambda^k(p)}}&+&\displaystyle\sum_{\scriptsize\begin{array}{c}p\in C(\overline{k+1})\cap A_{k+2}(f),\\ \nabla(\overline{k})(p)^{\vee} A_{k}^-(f)\end{array}\normalsize}{(-1)^{\overline{\lambda}^{k+1}(p)}};
\end{array}}}
\end{equation*} where $\lambda^k(p)$ is the Morse index of $L_a\circ f|_{A_k(f)}$ at $p$ and $\overline{\lambda}^{k+1}(p)$ is the Morse index of $L_a\circ f|_{\overline{A_{k+1}(f)}}$ at $p$. 
Similarly, 
\begin{equation*}\label{chik2}\renewcommand{\arraystretch}{1.8}{\setlength{\arraycolsep}{0.1cm}{\begin{array}{cclcl}
\chi(\overline{A_{k+2}^+(f)})&=&\displaystyle\sum_{p\in C({k+2}^+)}{(-1)^{\lambda^{k+2}(p)}}&+&\displaystyle\sum_{\scriptsize\begin{array}{c}p\in C(\overline{k+3})\cap A_{k+4}(f),\\ \nabla(\overline{k+2})(p)^{\vee} A_{k+2}^+(f)\end{array}\normalsize}{(-1)^{\overline{\lambda}^{k+3}(p)}};\\
\chi(\overline{A_{k+2}^-(f)})&=&\displaystyle\sum_{p\in C({k+2}^-)}{(-1)^{\lambda^{k+2}(p)}}&+&\displaystyle\sum_{\scriptsize\begin{array}{c}p\in C(\overline{k+3})\cap A_{k+4}(f),\\ \nabla(\overline{k+2})(p)^{\vee} A_{k+2}^-(f)\end{array}\normalsize}{(-1)^{\overline{\lambda}^{k+3}(p)}}.
\end{array}}}
\end{equation*}

In order to calculate the sum $\chi(\overline{A_{k}^+(f)})-\chi(\overline{A_{k}^-(f)})+\chi(\overline{A_{k+2}^+(f)})-\chi(\overline{A_{k+2}^-(f)})$ we will compare the parities of the Morse indices ${\overline{\lambda}^{k+1}(p)}$ and ${\lambda^{k+2}(p)}$, where $p\in A_{k+2}(f)$ is a critical point of $L_a\circ f|_{\overline{A_{k+1}(f)}}$. To do this, we consider the cases where $\nabla(L_a\circ f|_{\overline{A_{k}(f)}})(p)$ is pointing inwards $A_{k}^+(f)$ and $\nabla(L_a\circ f|_{\overline{A_{k}(f)}})(p)$ is pointing inwards $A_{k}^-(f)$.\\

%In order to study the sum $\chi(\overline{A_{k}^+(f)})-\chi(\overline{A_{k}^-(f)})+\chi(\overline{A_{k+2}^+(f)})-\chi(\overline{A_{k+2}^-(f)})$, let us compare the Morse indices ${\overline{\lambda}^{k+1}(p)}$ of the critical points $p\in C(\overline{k+1})\cap A_{k+2}(f)$ according to $\nabla(L_a\circ f|_{\overline{A_{k}(f)}})(p)$ is pointing inwards $A_{k}^+(f)$ or $\nabla(L_a\circ f|_{\overline{A_{k}(f)}})(p)$ is pointing inwards $A_{k}^-(f)$.\\

If $p\in A_{k+2}(f)$, then there are local coordinates systems $(x_1, \ldots, x_m)$ around $p$ and $(y_1, \ldots, y_n)$ around $f(p)$ such that 
\begin{equation}\label{formalocalak+2}
\begin{array}{l}
 y_i\circ f=x_i \text{, for } i\leq n-1,\\
 y_n\circ f=x_n^{k+3}+\displaystyle\sum_{i=1}^{k+1}x_ix_n^{k+2-i}+ x_{n+1}^2+\ldots+ x_{n+\lambda-1}^2-x_{n+\lambda}^2-\ldots-x_m^2.
\end{array}
\end{equation} Let us denote by $\gamma$ the function $y_n\circ f$. In a neighborhood of $p$, we have \begin{equation}\label{equacoesdosakscap2}\renewcommand{\arraystretch}{2.5}{\setlength{\arraycolsep}{0.1cm}{\begin{array}{ccl}
A_k(f)&=&\{x_{n+1}=\ldots=x_m=0;\ffrac{\partial^{j}\gamma}{\partial x_n^{j}}=0, j=1,\ldots,k; \ffrac{\partial^{k+1}\gamma}{\partial x_n^{k+1}}\neq0\};\\
\overline{A_k(f)}&=&\{x_{n+1}=\ldots=x_m=0;\ffrac{\partial^{j}\gamma}{\partial x_n^{j}}=0, j=1,\ldots,k\}.
\end{array}}}\end{equation}

In particular, around $p$, on the manifold $\overline{A_{k+1}(f)}$ we have $x_{n+1}=\ldots=x_m=0$ and for $j=1,\ldots,k+1$, \begin{equation*}
\ffrac{\partial^{j}\gamma}{\partial x_n^{j}}=0\Rightarrow \ffrac{(k+3)!}{(k+3-j)!}x_n^{k+3-j}+\displaystyle\sum_{i=1}^{k+2-j}{\ffrac{(k+2-i)!}{(k+2-j-i)!}x_ix_n^{k+2-j-i}}=0,
\end{equation*} then, for $j=1,\ldots,k+1$, the following holds
\begin{equation}\label{xrs}
x_{k+2-j}=-\binom{k+3}{j}x_n^{k+3-j}-\displaystyle\sum_{i=1}^{k+1-j}{\binom{k+2-i}{j}x_ix_n^{k+2-j-i}}.
\end{equation} 
Applying relations (\ref{xrs}) for $j=1,\ldots,k+1$, on $\overline{A_{k+1}(f)}$ we have
\begin{equation}\label{xrsak+1}
x_r=c_rx_n^{r+1}, \ \ r=1,\ldots, k+1,
\end{equation} for constants $c_r\in\R$. 
Thus, we consider $x_{k+2}, \ldots, x_n$ as local coordinates in $\overline{A_{k+1}(f)}$.

From Lemma \ref{Laproperties}, since $p\in A_{k+2}(f)$, then $p\notin C(L_a\circ f|_{\overline{A_k(f)}})$. 
Moreover, from Lemma \ref{lemaseparado}, if $p\in C(L_a\circ f|_{A_{k+2}(f)})$, then $p\in C(L_a\circ f|_{\overline{A_{k+1}(f)}})$. Therefore, the critical points of $L_a\circ f|_{A_{k+2}(f)}$ are correct critical points of $L_a\circ f|_{\overline{A_k(f)}}$. In particular, we have that $\nabla(L_a\circ f|_{\overline{A_k(f)}})(p)\neq\vec{0}$ and there exists $\eta(p)\in\R\setminus\{0\}$ such that $$\nabla(L_a\circ f|_{\overline{A_k(f)}})(p)=\eta(p).\nabla\left(\frac{\partial^{k+1}\gamma}{\partial x_n^{k+1}}\right)(p).$$ Thus, by the characterization of $A_{k+2}^+(f)$ and $A_{k+2}^-(f)$ from \cite[p.186]{Dutertrefukui} and by the characterization of $A_{k}^+(f)$ and $A_{k}^-(f)$ from \cite[Proposition 6.1, p.188]{Dutertrefukui} we have the following cases: If $p\in A_{k+2}^+(f)$, then $\lambda$ in the normal form (\ref{formalocalak+2}) is even and we have \begin{equation}\label{entraesaipar}\begin{array}{ccc}
\nabla(L_a\circ f|_{\overline{A_k(f)}})(p) \text{ is pointing inwards } A_{k}^+(f) &\Leftrightarrow& \eta(p)>0;\\
\nabla(L_a\circ f|_{\overline{A_k(f)}})(p) \text{ is pointing inwards } A_{k}^-(f)&\Leftrightarrow& \eta(p)<0.
\end{array}\end{equation} If $p\in A_{k+2}^-(f)$, then $\lambda$ in (\ref{formalocalak+2}) is odd and we have \begin{equation}\label{entraesaiimpar}\begin{array}{ccc}
\nabla(L_a\circ f|_{\overline{A_k(f)}})(p) \text{ is pointing inwards } A_{k}^+(f) &\Leftrightarrow& \eta(p)<0;\\
\nabla(L_a\circ f|_{\overline{A_k(f)}})(p) \text{ is pointing inwards } A_{k}^-(f) &\Leftrightarrow& \eta(p)>0.
\end{array}\end{equation} Since $\ffrac{\partial^{k+1}\gamma}{\partial x_n^{k+1}}(x)=\ffrac{(k+3)!}{2!}x_n^2+(k+1)!x_1$ and $x_n=0$ at $p\in A_{k+2}(f)$, then \begin{equation}\label{etapfatorial}\nabla(L_a\circ f|_{\overline{A_k(f)}})(p)=\left(\eta(p)(k+1)!,0, \ldots,0\right).
\end{equation}
%\begin{equation*}\setlength{\arraycolsep}{0.1cm}{\renewcommand{\arraystretch}{2.5}{\begin{array}{ccl}
%\nabla\left(\ffrac{\partial^{k+1}\gamma}{\partial x_n^{k+1}}\right)(x)&=&\left((k+1)!,0, \ldots,0,(k+3)!x_n,0, \ldots,0\right);\\
%\nabla\left(\ffrac{\partial^{k+1}\gamma}{\partial x_n^{k+1}}\right)(p)&=&\left((k+1)!,0, \ldots,0\right);
%\end{array}}}
%\end{equation*}  
Therefore, to analyse the sign of $\eta(p)$ it is enough to calculate $\ffrac{\partial(L_a\circ f)}{\partial x_1}(p)$. 

From (\ref{equacoesdosakscap2}), on $\overline{A_k(f)}$ we obtain \begin{equation}\label{xrsak}
x_{r}= c_n^{r}x_n^{r+1}+c^{r}_{1,n}x_1x_n^{r-1}, \ \ r=2,\ldots,k+1,
\end{equation} where $c_n^{r}, c^{r}_{1,n}\in\R$ are nonzero constants. Replacing (\ref{xrsak}) in the expression of $y_n\circ f(x)$, on $\overline{A_k(f)}$, the map $g(x)=(g_1(x),\ldots,g_n(x))$, defined by $g_j(x)=y_j\circ f(x)$, has its coordinate functions given by \begin{equation*}g_j(x)=\left\{
\begin{array}{ll}
x_j, & j=1,k+2,\ldots, n-1;\\
c_n^jx_n^{j+1}+c_{1,n}^{j}x_1x_n^{j-1},& j=2,\ldots,k+1;\\
c_nx_n^{k+3}+c_{1,n}x_1x_n^{k+1}, & j=n
\end{array}\right.
\end{equation*} where $c_n, c_{1,n}\in\R$ are nonzero constants. Since $L_a\circ f(x)=\displaystyle\sum_{i=1}^{n}{a_i\left(\psi_i\circ g\right)(x)}$, then \begin{equation}\label{partial1}
\ffrac{\partial(L_a\circ f)}{\partial x_1}(x)=\displaystyle\sum_{i=1}^{n}{a_i\left(\displaystyle\sum_{j=1}^{n}{\ffrac{\partial \psi_i}{\partial y_j}(g(x))\ffrac{\partial g_j}{\partial x_1}(x)}\right)}
\end{equation} and, evaluating at $p$, $\ffrac{\partial(L_a\circ f)}{\partial x_1}(p)=\sum_{i=1}^{n}{a_i\ffrac{\partial \psi_i}{\partial y_1}(g(p))}$. Thereby, by equation (\ref{etapfatorial}),  
\begin{equation}\label{relacaoeta}
\displaystyle\sum_{i=1}^{n}{a_i\ffrac{\partial \psi_i}{\partial y_1}(g(p))}=\eta(p).(k+1)!\end{equation} and we have to analyse the expression $\displaystyle\sum_{i=1}^{n}{a_i\ffrac{\partial \psi_i}{\partial y_1}(g(p))}$. To do this, let us consider some informations about the Hessian matrix of $L_a\circ f|_{\overline{A_{k+1}(f)}}$ at the point $p$. 

By relations (\ref{xrsak+1}), we know that, locally, on $\overline{A_{k+1}(f)}$, the coordinate functions of the map $g(x)=(g_1(x),\ldots,g_n(x))$ are given by 
\begin{equation}\label{gsobreAk1}
g_j(x)=\left\{\begin{array}{ll}
c_jx_n^{j+1}, & j=1,\ldots,k+1;\\
x_j, & j=k+2,\ldots, n-1;\\
c_nx_n^{k+3}, & j=n;
\end{array}\right.\end{equation} where $c_n\in\R$ is a nonzero constant obtained by replacing relations ($\ref{xrsak+1}$) in the expression of $y_n\circ f(x)$. Recall the notation $L_a\circ f=\displaystyle\sum_{i=1}^{n}{a_i\left(\psi_i\circ g\right)(x)}$ and consider its partial derivatives of order two evaluated at $p\in A_{k+2}(f)$, by relations (\ref{gsobreAk1}) and their partial derivatives, and since $x_n=0$ at $p\in A_{k+2}(f)$ it is not difficult to see that \begin{equation}\label{partialxndois}\ffrac{\partial^{2}\left(L_a\circ f\right)}{\partial x_n^2}(p)=2c_1\displaystyle\sum_{i=1}^{n}{a_i\ffrac{\partial \psi_i}{\partial y_1}(g(p))}.\end{equation}

By equations (\ref{equacoesdosakscap2}) and (\ref{xrsak+1}), we know that, around $p$, $x_{k+2}, \ldots, x_n$ are local coordinates in $\overline{A_{k+1}(f)}$ and $x_{k+2}, \ldots, x_{n-1}$ are local coordinates in $A_{k+2}(f)$. Thus, by equation (\ref{partial1}) evaluated at $p$ and (\ref{partialxndois}), the Hessian matrix of $L_a\circ f|_{\overline{A_{k+1}(f)}}$ at the point $p$ is given by \begin{equation}\label{matrix}\left[\setlength{\arraycolsep}{0.1cm}{\begin{array}{ccccc} 
\multicolumn{2}{c}{\multirow{2}{*}{$\left[\begin{array}{c}\ffrac{\partial^{(2)} (L_a\circ f)}{\partial x_s \partial x_{\l}}(p)
\end{array}\right]_{k+2\leq s,\l\leq n-1}$}} & \vdots & \multicolumn{2}{c}{\multirow{2}{*}{$O_{(n-k-2)\times1}$}}\\ 
& &\vdots & & \\
\multicolumn{2}{c}{\cdots \ \cdots \ \cdots \ \cdots \ \cdots \ \cdots \ \cdots \ \cdots} & \vdots & \multicolumn{2}{c}{\cdots \ \cdots \ \cdots \ \cdots \ \cdots} \\
\multicolumn{2}{c}{O_{1\times(n-k-2)}}& \vdots & \multicolumn{2}{c}{2c_1\ffrac{\partial(L_a\circ f)}{\partial x_1}(p)}
\end{array}}\right]\end{equation} where $O_{(n-k-2)\times1}$ and $O_{1\times(n-k-2)}$ denote null submatrices and the submatrix $$\left[\begin{array}{c}
\ffrac{\partial^{(2)} (L_a\circ f)}{\partial x_s \partial x_{\l}}(p)
\end{array}\right]_{k+2\leq s,\l\leq n-1}$$ is the Hessian matrix of $L_a\circ f|_{A_{k+2}(f)}$ evaluated at $p$. 

We can see from equation (\ref{xrsak+1}) that $c_1=-\frac{(k+3)!}{2!(k+1)!}<0$, so that $$\s\left(Hess(L_a\circ f|_{\overline{A_{k+1}}})(p)\right)=-\s\left(Hess(L_a\circ f|_{A_{k+2}})(p)\right).\s\eta(p),$$ that is, \begin{equation}\s\eta(p)=-(-1)^{\overline{\lambda}^{k+1}(p)}.(-1)^{\lambda^{k+2}(p)}.\end{equation} Thus, according to equivalences (\ref{entraesaipar}) and (\ref{entraesaiimpar}), we conclude that: If $p\in A_{k+2}^+(f)$, then $\lambda$ in the expression (\ref{formalocalak+2}) is even, and
\begin{equation*}\begin{array}{lll}
\nabla(L_a\circ f|_{\overline{A_k(f)}})(p) \text{ points inwards } A_{k}^{+}(f) \ &\Leftrightarrow& \, \eta(p)>0\\
	%&\Leftrightarrow & (-1)^{\overline{\lambda}^{k+1}(p)}.(-1)^{\lambda^{k+2}(p)}<0\\
	&\Leftrightarrow & \, \overline{\lambda}^{k+1}(p)\equiv 1+\lambda^{k+2}(p)\mod 2\\
\\
\nabla(L_a\circ f|_{\overline{A_k(f)}})(p) \text{ points inwards } A_{k}^{-}(f) \ &\Leftrightarrow& \, \eta(p)<0\\
	%&\Leftrightarrow & (-1)^{\overline{\lambda}^{k+1}(p)}.(-1)^{\lambda^{k+2}(p)}>0\\
	& \Leftrightarrow & \, \overline{\lambda}^{k+1}(p)\equiv\lambda^{k+2}(p)\mod 2
\end{array}\end{equation*} If $p\in A_{k+2}^-(f)$, then $\lambda$ in the expression (\ref{formalocalak+2}) is odd, and 
\begin{equation*}\begin{array}{lll} 
\nabla(L_a\circ f|_{\overline{A_k(f)}})(p) \text{ points inwards } A_{k}^{+}(f) \ &\Leftrightarrow& \ \eta(p)<0\\
	%&\Leftrightarrow & (-1)^{\overline{\lambda}^{k+1}(p)}.(-1)^{\lambda^{k+2}(p)}>0\\
	&\Leftrightarrow & \ \overline{\lambda}^{k+1}(p)\equiv \lambda^{k+2}(p)\mod 2\\
\\
\nabla(L_a\circ f|_{\overline{A_k(f)}})(p) \text{ points inwards } A_{k}^{-}(f) \ &\Leftrightarrow & \ \eta(p)>0\\
	%&\Leftrightarrow & (-1)^{\overline{\lambda}^{k+1}(p)}.(-1)^{\lambda^{k+2}(p)}<0\\
	&\Leftrightarrow & \ \overline{\lambda}^{k+1}(p)\equiv 1+\lambda^{k+2}(p)\mod 2
\end{array}\end{equation*}

Using these equivalences and keeping the notations (\ref{notacaogradientes}), we obtain: 

\begin{equation*}{\small\setlength{\arraycolsep}{0.1cm}{\begin{array}{cccc}
&\displaystyle\sum_{\scriptsize\begin{array}{c}p\in C(\overline{k+1})\cap A_{k+2}(f),\\ \nabla(\overline{k})(p)^{\vee}A_{k}^+(f)\end{array}}{(-1)^{\overline{\lambda}^{k+1}(p)}}&-&\displaystyle\sum_{{\scriptsize\begin{array}{c}p\in C(\overline{k+1})\cap A_{k+2}(f),\\ \nabla(\overline{k})(p)^{\vee}A_{k}^-(f)\end{array}}}{(-1)^{\overline{\lambda}^{k+1}(p)}}\\
\\
=&\displaystyle\sum_{\scriptsize\begin{array}{c}p\in C(\overline{k+1})\cap A_{k+2}^+(f),\\ \nabla(\overline{k})(p)^{\vee}A_{k}^+(f)\end{array}}{(-1)^{\overline{\lambda}^{k+1}(p)}}&+&\displaystyle\sum_{\scriptsize\begin{array}{c}p\in C(\overline{k+1})\cap A_{k+2}^-(f),\\ \nabla(\overline{k})(p)^{\vee} A_{k}^+(f)\end{array}}{(-1)^{\overline{\lambda}^{k+1}(p)}}\\
\\
-&\displaystyle\sum_{\scriptsize\begin{array}{c}p\in C(\overline{k+1})\cap A_{k+2}^+(f),\\ \nabla(\overline{k})(p)^{\vee}A_{k}^-(f)\end{array}}{(-1)^{\overline{\lambda}^{k+1}(p)}}&-&\displaystyle\sum_{\scriptsize\begin{array}{c}p\in C(\overline{k+1})\cap A_{k+2}^-(f),\\ \nabla(\overline{k})(p)^{\vee} A_{k}^-(f)\end{array}}{(-1)^{\overline{\lambda}^{k+1}(p)}}\\
\\
=&-\displaystyle\sum_{{\scriptsize\begin{array}{c}p\in C({k+2}^+),\\ \nabla(\overline{k})(p)^{\vee}A_{k}^+(f)\end{array}}}{(-1)^{\lambda^{k+2}(p)}}&+&\displaystyle\sum_{\scriptsize\begin{array}{c}p\in C({k+2}^-),\\ \nabla(\overline{k})(p)^{\vee}A_{k}^+(f)\end{array}}{(-1)^{\lambda^{k+2}(p)}}\\
\\
-&\displaystyle\sum_{{\scriptsize\begin{array}{c}p\in C({k+2}^+),\\ \nabla(\overline{k})(p)^{\vee}A_{k}^-(f)\end{array}}}{(-1)^{\lambda^{k+2}(p)}}&+&\displaystyle\sum_{{\scriptsize\begin{array}{c}p\in C({k+2}^-),\\ \nabla(\overline{k})(p)^{\vee}A_{k}^-(f)\end{array}}}{(-1)^{\lambda^{k+2}(p)}}\\
\\
=&\displaystyle\sum_{ {\scriptsize\begin{array}{c}p\in C({k+2}^-)\end{array}}}{(-1)^{\lambda^{k+2}(p)}}&-&\displaystyle\sum_{{\scriptsize\begin{array}{c}p\in C({k+2}^+)\end{array}}}{(-1)^{\lambda^{k+2}(p)}}
\end{array}}}
\end{equation*}

Thus, \begin{equation*}\renewcommand{\arraystretch}{2.2}{\begin{array}{lllllll}
\chi(\overline{A_k^+(f)})-\chi(\overline{A_k^-(f)})&=&\displaystyle\sum_{p\in C(k^+)}(-1)^{\lambda^{k}(p)}&-&\displaystyle\sum_{p\in C({k}^-)}{(-1)^{\lambda^{k}(p)}}\\
&-&\displaystyle\sum_{p\in C(k+2^+)}{(-1)^{\lambda^{k+2}(p)}}&+&\displaystyle\sum_{p\in C(k+2^-)}(-1)^{\lambda^{k+2}(p)},
\end{array}}\end{equation*} for all $k=1,\ldots, n-4$, if $n$ is odd and for all $k=1,\ldots, n-3$, if $n$ is even.

If $n$ is even, from Theorem \ref{teobordo} and from Lemma \ref{lemadaperturbacao}, we have that $$\chi(\overline{A_{n-1}^+(f)})-\chi(\overline{A_{n-1}^-(f)})=\displaystyle\sum_{p\in C(n-1^+)}(-1)^{\lambda^{n-1}(p)}-\displaystyle\sum_{p\in C(n-1^-)}(-1)^{\lambda^{n-1}(p)}.$$ Therefore, if $n$ is even, we can conclude that {\small$$\displaystyle\sum_{k: \text{ odd}}{\left[\chi(\overline{A_k^+(f)})-\chi(\overline{A_k^-(f)})\right]}=\displaystyle\sum_{{\scriptsize\begin{array}{c}p\in C(1^+)\end{array}}}{(-1)^{\lambda^{1}(p)}}-\displaystyle\sum_{{\scriptsize\begin{array}{c}p\in C({1}^-)\end{array}}}{(-1)^{\lambda^{1}(p)}}.$$}

In the case that $n$ is odd, it remains to examine the sum $$\chi(\overline{A_{n-2}^+(f)})-\chi(\overline{A_{n-2}^-(f)})+\chi(\overline{A_n^+(f)})-\chi(\overline{A_n^-(f)}).$$ To do this, let us analyse the Morse index $\overline{\lambda}^{n-1}(p)$ of $L_a\circ f|_{\overline{A_{n-1}(f)}}$ at a point $p\in A_n(f)$. Suppose that $p\in A_n(f)$, then there exist local coordinates around $p$ and $f(p)$ such that
\begin{equation}\label{lambdadoan}
\begin{array}{l}
 y_i\circ f=x_i \text{, for } i\leq n-1,\\
 y_n\circ f=x_n^{n+1}+\displaystyle\sum_{i=1}^{n-1}x_ix_n^{n-i}+ x_{n+1}^2+\ldots+ x_{n+\lambda-1}^2-x_{n+\lambda}^2-\ldots-x_m^2
\end{array}
\end{equation} We set $\gamma:=y_n\circ f$, then around $p$ we can describe:
\begin{equation}\label{defaks}\setlength{\arraycolsep}{0.1cm}{\renewcommand{\arraystretch}{2.2}{\begin{array}{lll}
A_k(f)&=&\left\{x_{n+1}=\ldots=x_m=0;\ffrac{\partial^{j} \gamma}{\partial x_n^j}=0; j=1,\ldots,k;\ffrac{\partial^{k+1} \gamma}{\partial x_n^{k+1}}\neq0\right\};\\
\overline{A_k(f)}&=&\left\{x_{n+1}=\ldots=x_m=0;\ffrac{\partial^{j} \gamma}{\partial x_n^j}=0; j=1,\ldots,k\right\}.
\end{array}}}
\end{equation} In particular, we have that $A_n(f)=\{x_1=\ldots =x_m=0\}.$ Since $p\in A_n(f)$, from Lemma \ref{Laproperties}, $p$ is a non-degenerate critical point of $L_a\circ f|_{\overline{A_{n-1}(f)}}$ and $p\notin C(L_a\circ f|_{\overline{A_{n-2}(f)}})$. Thus, $p$ is a correct critical point of $L_a\circ f|_{\overline{A_{n-2}(f)}}$ and there exists $\xi(p)\in\R\setminus\{0\}$ such that \begin{equation*}\nabla(L_a\circ f|_{\overline{A_{n-2}(f)}})(p)=\xi(p).\nabla\left(\ffrac{\partial^{n-1} \gamma}{\partial x_n^{n-1}}\right)(p).\end{equation*} Since $\ffrac{\partial^{n-1} \gamma}{\partial x_n^{n-1}}(x)=\ffrac{(n+1)!}{2}x_n^2+(n-1)!x_1$ and $x_n=0$ at $p$, we have that $$\nabla(L_a\circ f|_{\overline{A_{n-2}(f)}})(p)=\left(\xi(p)(n-1)!,0,\ldots,0\right).$$ Hence, similarly to the sign of $\eta(p)$, to analyse the sign of $\xi(p)$ it is enough to calculate $\ffrac{\partial(L_a\circ f)}{\partial x_1}(p)$ restricted to $\overline{A_{n-2}{(f)}}$. By (\ref{defaks}), on $\overline{A_{n-2}{(f)}}$, around $p$, we have \begin{equation}\label{defxrs}\ffrac{\partial^{j} \gamma}{\partial x_n^j}=0\Rightarrow x_{n-j}=c_n^{n-j}x_n^{n-j+1}+c_{1,n}^{n-j}x_1x_n^{n-j-1}, \ j=1,\ldots,n-2,\end{equation}
so that, on $\overline{A_{n-2}(f)}$, we can write
\begin{equation}\label{defxrsii} x_{r}=c_n^{r}x_n^{r+1}+c_{1,n}^{r}x_1x_n^{r-1}, \ r=2,\ldots,n-1;\end{equation}
where $c_n^{r},c_{1,n}^{r}\in\R$ are nonzero constants obtained by (\ref{defxrs}). In this case, the map $g(x)=(g_1(x), \ldots, g_n(x))$ defined by $g_j(x)=y_j\circ f$ has its coordinate functions given by
\begin{equation*}g_j(x)=\left\{\begin{array}{ll}
x_1, & j=1;\\
c_n^{j}x_n^{j+1}+c_{1,n}^{j}x_1x_n^{j-1}, & j=2,\ldots,n;\\
\end{array}\right.\end{equation*} where $c_n^n,c_{1,n}^n\in\R$ are nonzero constants obtained by replacing the relations (\ref{defxrsii}) in the expression of $y_n\circ f$. Then, we use again expression (\ref{partial1}) to verify that $$\ffrac{\partial(L_a\circ f)}{\partial x_1}(p)=\sum_{i=1}^{n}{a_i\ffrac{\partial \psi_i}{\partial y_1}(g(p))}.$$

%Since $L_a\circ f(x)=\displaystyle\sum_{i=1}^{n}{a_i\left(\psi_i\circ g\right)(x)}$, then \begin{equation}\label{partial}
%\ffrac{\partial(L_a\circ f)}{\partial x_1}(x)=\displaystyle\sum_{i=1}^{n}{a_i\left(\displaystyle\sum_{j=1}^{n}{\ffrac{\partial \psi_i}{\partial y_j}(g(x))\ffrac{\partial g_j}{\partial x_1}(x)}\right)}
%\end{equation} and, evaluating at $p$, we have $\ffrac{\partial(L_a\circ f)}{\partial x_1}(p)=\sum_{i=1}^{n}{a_i\ffrac{\partial \psi_i}{\partial y_1}(g(p))}$.

We know by (\ref{defaks}) that around $p$ on $\overline{A_{n-1}(f)}$ we have $x_{n+1}=\ldots=x_{m}=0$ and $\ffrac{\partial^{j} \gamma}{\partial x_n^j}=0$, for $ j=1,\ldots,n-1$. Then the variables $x_1,\ldots,x_{n-1}$  can be expressed in terms of the variable $x_n$. Indeed, from (\ref{defxrs}), we obtain 
\begin{equation}\label{defxrsan1} x_{r}=c_n^{r}x_n^{r+1}+c_{1,n}^{r}x_1x_n^{r-1}, \ r=2,\ldots,n-1;
\end{equation} moreover, $\ffrac{\partial^{n-1} \gamma}{\partial x_n^{n-1}}=0\Rightarrow x_1=c_n^1x_n^2$, where $c^1_n=-\ffrac{(n+1)!}{2!(n-1)!}$. Then, 
\begin{equation}\label{deffinalxrs} 
x_{r}=C_rx_n^{r+1}, \ r=2,\ldots,n-1,
\end{equation} where $C_r\in\R$ are non-zero constants obtained by replacing $x_1=c_n^1x_n^2$ in (\ref{defxrsan1}). Therefore, on $\overline{A_{n-1}(f)}$ the coordinate functions of $g(x)=(g_1(x),\ldots,g_n(x))$ are given by
\begin{equation}\label{gsobreAn1} g_j(x)=C_jx_n^{j+1}, \, j=1,\ldots,n
\end{equation} where $C_1=c_n^1$ and $C_n\in\R$ is the non-zero constant obtained by replacing $x_1=c_n^1x_n^2$ and (\ref{deffinalxrs}) in the expression of $y_n\circ f$.

Recall the notation $L_a\circ f=\displaystyle\sum_{i=1}^{n}{a_i\left(\psi_i\circ g\right)(x)}$ and consider its partial derivatives of order two evaluated at $p\in A_{n}(f)$, by relations (\ref{gsobreAn1}) and their partial derivatives, and since $x_n=0$ at $p\in A_{n}(f)$ it is not difficult to see that 
%Remember that in equations from \ref{partial} we have 
%\begin{equation*}
%\ffrac{\partial L_a\circ f}{\partial x_{n}}(x)=\displaystyle\sum_{i=1}^{n}a_i\left(\displaystyle\sum_{j=1}^{n}\ffrac{\partial \psi_i}{\partial y_j}(g(x))\ffrac{\partial g_j}{\partial x_n}(x)\right)
%\end{equation*} Thus, 
%\begin{equation*}\setlength{\arraycolsep}{0.1cm}{\renewcommand{\arraystretch}{2.5}{\begin{array}{ccl}
%\ffrac{\partial^{(2)} (L_a\circ f)}{\partial x_n^2}(x)&=&\displaystyle\sum_{i=1}^{n}a_i\displaystyle\sum_{j=1}^{n}\left[{\left(\displaystyle\sum_{r=1}^{n}\ffrac{\partial^{(2)} \psi_i}{\partial y_r \partial y_j}(g(x))\ffrac{\partial g_r}{\partial x_n}(x)\right)\ffrac{\partial g_j}{\partial x_{n}}(x)}\right.\\
%&+&\left.\ffrac{\partial \psi_i}{\partial y_j}(g(x))\ffrac{\partial^{(2)} g_j}{\partial x_n^2}(x)\right].
%\end{array}}}\end{equation*} Then, on $\overline{A_{n-1}(f)}$,
%\begin{equation*}\setlength{\arraycolsep}{0.1cm}{\renewcommand{\arraystretch}{2.5}{\begin{array}{ccl}
%\ffrac{\partial^{(2)} (L_a\circ f)}{\partial x_n^2}(x)&=&\displaystyle\sum_{i=1}^{n}a_i\displaystyle\sum_{j=1}^{n}\left[{\left(\displaystyle\sum_{r=1}^{n}\ffrac{\partial^{(2)} \psi_i}{\partial y_r \partial y_j}(g(x))(r+1)C_rx_n^{r}\right)(j+1)C_jx_n^{j}}\right.\\
%&+&\left.\ffrac{\partial \psi_i}{\partial y_j}(g(x))j(j+1)C_jx_n^{j-1}\right].
%\end{array}}}\end{equation*} and as $x_n=0$ on $A_n(f)$, we have 
$$\ffrac{\partial^{(2)} (L_a\circ f)}{\partial x_n^2}(p)=2C_1\displaystyle\sum_{i=1}^{n}a_i\ffrac{\partial \psi_i}{\partial y_1}(g(p))$$ where $C_1=c_n^1<0$. Therefore,
$$\setlength{\arraycolsep}{0.1cm}{\begin{array}{lll}
\s Hess(L_a\circ f|_{\overline{A_{n-1}(f)}})(p)&=&-\s\ffrac{\partial L_a\circ f}{\partial x_1}(p)=-\s\xi(p),\\
\end{array}}
$$ that is, $$\s\xi(p)=-(-1)^{\overline{\lambda}^{n-1}(p)}.$$ 

By the characterization of $A_{n}^+(f)$ and $A_{n}^-(f)$ from \cite[p.186]{Dutertrefukui} and by the characterization of $A_{n-2}^+(f)$ and $A_{n-2}^-(f)$ from \cite[Proposition 6.1, p.188]{Dutertrefukui}, we have: If $p\in A_{n}^+(f)$, then $\lambda$ in the expression (\ref{lambdadoan}) is even and we say that
\begin{equation*}\begin{array}{lll}\nabla(L_a\circ f|_{\overline{A_{n-2}(f)}})(p) \text{ is pointing inwards } A_{n-2}^{+}(f)&\Leftrightarrow& \xi(p)>0;\\
\nabla(L_a\circ f|_{\overline{A_{n-2}(f)}})(p) \text{ is pointing inwards } A_{n-2}^{-}(f)&\Leftrightarrow& \xi(p)<0.
\end{array}\end{equation*}
If $p\in A_{n}^-(f)$, then $\lambda$ in the expression (\ref{lambdadoan}) is odd and we say that
\begin{equation*}\begin{array}{lll}\nabla(L_a\circ f|_{\overline{A_{n-2}(f)}})(p) \text{ is pointing inwards } A_{n-2}^{+}(f)&\Leftrightarrow& \xi(p)<0;\\
\nabla(L_a\circ f|_{\overline{A_{n-2}(f)}})(p) \text{ is pointing inwards } A_{n-2}^{-}(f)&\Leftrightarrow& \xi(p)>0.
\end{array}\end{equation*}
Thus, if $p\in A_{n}^+(f)$, then $\lambda$ is even in the expression (\ref{lambdadoan}) and 

\begin{equation*}\begin{array}{lll}\nabla(L_a\circ f|_{\overline{A_{n-2}(f)}})(p) \text{ is pointing inwards } A_{n-2}^{+}(f)&\Leftrightarrow& \ \xi(p)>0\\
	%&\Leftrightarrow & (-1)^{\overline{\lambda}^{n-1}(p)}<0\\
	%&\Leftrightarrow &\overline{\lambda}^{n-1}(p) \text{ is odd }\\
	&\Leftrightarrow & \ \overline{\lambda}^{n-1}(p)\equiv 1\mod 2\\
	\\
\nabla(L_a\circ f|_{\overline{A_{n-2}(f)}})(p) \text{ is pointing inwards } A_{n-2}^{-}(f)&\Leftrightarrow& \ \xi(p)<0\\
	%&\Leftrightarrow & (-1)^{\overline{\lambda}^{n-1}(p)}>0\\
	%&\Leftrightarrow & \overline{\lambda}^{n-1}(p) \text{ is even }\\
		&\Leftrightarrow & \ \overline{\lambda}^{n-1}(p)\equiv 0\mod 2
\end{array}\end{equation*} If $p\in A_{n}^-(f)$, then $\lambda$ in the expression (\ref{lambdadoan}) is odd and	
	\begin{equation*}\begin{array}{lll}\nabla(L_a\circ f|_{\overline{A_{n-2}(f)}})(p) \text{ is pointing inwards } A_{n-2}^{+}(f)& \Leftrightarrow& \ \xi(p)<0\\
	%&\Leftrightarrow & (-1)^{\overline{\lambda}^{n-1}(p)}>0\\
	%&\Leftrightarrow & \overline{\lambda}^{n-1}(p) \text{ is even }\\
	&\Leftrightarrow & \ \overline{\lambda}^{n-1}(p)\equiv 0\mod 2\\
\\
\nabla(L_a\circ f|_{\overline{A_{n-2}(f)}})(p) \text{ is pointing inwards } A_{n-2}^{-}(f)&\Leftrightarrow& \ \xi(p)>0\\
	%&\Leftrightarrow & (-1)^{\overline{\lambda}^{n-1}(p)}<0\\
	%&\Leftrightarrow & \overline{\lambda}^{n-1}(p) \text{ is odd }\\
	&\Leftrightarrow & \ \overline{\lambda}^{n-1}(p)\equiv 1\mod 2
\end{array}\end{equation*}

Using these equivalences and keeping the notations (\ref{notacaogradientes}), we obtain:

\begin{equation*}{\small\setlength{\arraycolsep}{0.1cm}{\begin{array}{cccc}
&\displaystyle\sum_{\scriptsize\begin{array}{c}p\in C(\overline{{n-1}})\cap A_{n}(f),\\ \nabla(\overline{n-2})(p)^{\vee}A_{n-2}^+(f)\end{array}}{(-1)^{\overline{\lambda}^{n-1}(p)}}&-&\displaystyle\sum_{\scriptsize\begin{array}{c}p\in C(\overline{{n-1}})\cap A_{n}(f),\\ \nabla(\overline{n-2})(p)^{\vee}A_{n-2}^-(f)\end{array}}{(-1)^{\overline{\lambda}^{n-1}(p)}}\\
\\
=&\displaystyle\sum_{\scriptsize\begin{array}{c}p\in C(\overline{{n-1}})\cap A_{n}^+(f),\\ \nabla(\overline{n-2})(p)^{\vee}A_{n-2}^+(f)\end{array}}{(-1)^{\overline{\lambda}^{n-1}(p)}}&+&\displaystyle\sum_{\scriptsize\begin{array}{c}p\in C(\overline{{n-1}})\cap A_{n}^-(f),\\ \nabla(\overline{n-2})(p)^{\vee}A_{n-2}^+(f)\end{array}}{(-1)^{\overline{\lambda}^{n-1}(p)}}\\
\end{array}}}
\end{equation*}

\begin{equation*}{\small\setlength{\arraycolsep}{0.1cm}{\begin{array}{cccc}
-&\displaystyle\sum_{\scriptsize\begin{array}{c}p\in C(\overline{{n-1}})\cap A_{n}^+(f),\\ \nabla(\overline{n-2})(p)^{\vee} A_{n-2}^-(f)\end{array}}{(-1)^{\overline{\lambda}^{n-1}(p)}}&-&
\displaystyle\sum_{\scriptsize\begin{array}{c}p\in C(\overline{{n-1}})\cap A_{n}^-(f),\\ \nabla(\overline{n-2})(p)^{\vee} A_{n-2}^-(f)\end{array}}{(-1)^{\overline{\lambda}^{n-1}(p)}}\\
\\
=&\displaystyle\sum_{\scriptsize\begin{array}{c}p\in C(\overline{{n-1}})\cap A_{n}^+(f),\\ \nabla(\overline{n-2})(p)^{\vee}A_{n-2}^+(f)\end{array}}{(-1)}&+&\displaystyle\sum_{\scriptsize\begin{array}{c}p\in C(\overline{{n-1}})\cap A_{n}^-(f),\\ \nabla(\overline{n-2})(p)^{\vee}A_{n-2}^+(f)\end{array}}{(\, 1\,)}\\
\\
-&\displaystyle\sum_{\scriptsize\begin{array}{c}p\in C(\overline{{n-1}})\cap A_{n}^+(f),\\ \nabla(\overline{n-2})(p)^{\vee} A_{n-2}^-(f)\end{array}}{(\, 1\,)}&-&
\displaystyle\sum_{\scriptsize\begin{array}{c}p\in C(\overline{{n-1}})\cap A_{n}^-(f),\\ \nabla(\overline{n-2})(p)^{\vee} A_{n-2}^-(f)\end{array}}{(-1)}\\
\\
=&\displaystyle\sum_{\scriptsize\begin{array}{c}p\in C(\overline{{n-1}})\cap A_{n}^+(f),\\ \nabla(\overline{n-2})(p)^{\vee}A_{n-2}^+(f)\end{array}}{(-1)}&-&\displaystyle\sum_{\scriptsize\begin{array}{c}p\in C(\overline{{n-1}})\cap A_{n}^-(f),\\ \nabla(\overline{n-2})(p)^{\vee}A_{n-2}^+(f)\end{array}}{(-1)}\\
\\
+&\displaystyle\sum_{\scriptsize\begin{array}{c}p\in C(\overline{{n-1}})\cap A_{n}^+(f),\\ \nabla(\overline{n-2})(p)^{\vee} A_{n-2}^-(f)\end{array}}{(-1)}&-&
\displaystyle\sum_{\scriptsize\begin{array}{c}p\in C(\overline{{n-1}})\cap A_{n}^-(f),\\ \nabla(\overline{n-2})(p)^{\vee} A_{n-2}^-(f)\end{array}}{(-1)}\\
\\
=&-\#A_n^+(f)+\#A_n^-(f).&&
\end{array}}}
\end{equation*} Then, the sum $\chi(\overline{A_{n-2}^+(f)})-\chi(\overline{A_{n-2}^-(f)})$ is equal to 

\small $$\displaystyle\sum_{p\in C(n-2^+)}{(-1)^{\lambda^{n-2}(p)}}-\displaystyle\sum_{\tiny p\in C({n-2}^-)}(-1)^{\lambda^{n-2}(p)}-\#A_n^+(f)+\#A_n^-(f).$$ \normalsize Moreover, $$\chi(\overline{A_n^+(f)})-\chi(\overline{A_n^-(f)})=\#A_n^+(f)-\#A_n^-(f).$$ Hence, in the case that $n$ is odd we also have that

{\small$$\displaystyle\sum_{k: \text{ odd}}{\left[\chi(\overline{A_k^+(f)})-\chi(\overline{A_k^-(f)})\right]}=\displaystyle\sum_{{\scriptsize\begin{array}{c}p\in C(1^+)\end{array}}}{(-1)^{\lambda^{1}(p)}}-\displaystyle\sum_{{\scriptsize\begin{array}{c}p\in C({1}^-)\end{array}}}{(-1)^{\lambda^{1}(p)}}.$$}

Finally, by equation (\ref{ladoesquerdo}) we have {\small\begin{equation*}\renewcommand{\arraystretch}{2.3}{\setlength{\arraycolsep}{0.1cm}{
\begin{array}{lllll}
\chi(M)&=&\displaystyle\sum_{p\in A_1^+(f)\cap C}(-1)^{\lambda^1(p)}&-&\displaystyle\sum_{p\in A_1^-(f)\cap C}(-1)^{\lambda^1(p)}\\
&=&\displaystyle\sum_{{\scriptsize\begin{array}{c}p\in C(1^+)\end{array}}}{(-1)^{\lambda^{1}(p)}}&-&\displaystyle\sum_{{\scriptsize\begin{array}{c}p\in C({1}^-)\end{array}}}{(-1)^{\lambda^{1}(p)}}.
\end{array}}}
\end{equation*}} Therefore, $\chi(M)=\displaystyle\sum_{k: \text{ odd}}{\left[\chi(\overline{A_k^+(f)})-\chi(\overline{A_k^-(f)})\right]}.$ \end{proof}

\bibliographystyle{plain}
\bibliography{refbiblioNewproof}

\end{document}